\documentclass{article}

\usepackage{amssymb,amsbsy,amsmath,amsfonts,amssymb,amscd,amsthm}
\usepackage{graphicx}
\usepackage{xcolor}
\usepackage{cancel}

\newtheorem{dummy}{dummy}[section]
\newtheorem{Theorem}[dummy]{Theorem}
\newtheorem{proposition}[dummy]{Proposition}
\newtheorem*{Proposition*}{Proposition}
\newtheorem{lemma}[dummy]{Lemma}
\newtheorem{Corollary}[dummy]{Corollary}

\newtheorem{defi}[dummy]{Definition}
\newtheorem{Remark}[dummy]{Remark}

\usepackage{hyperref}

\usepackage{ulem}

\usepackage{subcaption}

\title{On the quadratic stability of asymmetric Hermite basis with  application to plasma physics with oscillating electric field}

\author{
R. Dai%
\thanks{\url{ruiyang.dai@sorbonne-universite.fr}; Laboratoire Jacques-Louis Lions, Sorbonne Université, 4 Place Jussieu, 75005, Paris, France; Corresponding author}
\ and B. Despr\'es%
\thanks{\url{bruno.despres@sorbonne-universite.fr}; Laboratoire Jacques-Louis Lions, Sorbonne Université, 4 Place Jussieu, 75005, Paris, France}
}
\date{}

\begin{document}
\maketitle

\begin{abstract}
We analyze why the discretization of linear transport with   asymmetric Hermite basis functions  can be  instable in quadratic norm.
The main reason is that the finite truncation of the infinite moment linear system looses the skew-symmetry property with respect to the 
Gram matrix.
Then we propose an original closed formula for the scalar product of any pair of asymmetric basis functions. It makes possible the construction of two simple modifications of the linear systems which recover the skew-symmetry property.
By construction the new methods are quadratically stable with respect to the natural $L^2$ norm. We explain how to generalize to other transport equations
encountered in numerical plasma physics. Basic numerical tests with oscillating electric fields of different nature  illustrate the unconditional stability properties of our algorithms.
\end{abstract}

\section{Introduction}

 Asymmetric Hermite basis are widely used
for the numerical discretization of transport phenomenon in plasma physics, since
they  are amenable for the preservation of natural invariants such as the total mass or the total energy.
However they are not symmetric which, on mathematical grounds,  means that they do not constitute an orthonormal family of the space $L^2(\mathbb R)$ endowed with the scalar product
\begin{equation} \label{eq:sp}
(\alpha, \beta)=\int_\mathbb R \alpha(v) \beta(v)dv, \qquad \alpha,\beta\in L^2(\mathbb R).
\end{equation}
In consequence they can trigger numerical instabilities in numerical methods.
The aim of our work is to explain how to recover the hidden quadratic stability  	of  asymmetric Hermite basis, which turns into new and stable numerical methods.
Mathematically this is based on exact formulas for the calculation of the scalar product of two asymmetric functions. To the best of our knowledge, these
 formulas are original with respect to the huge literature on special functions \cite{nist,magnus}.

In plasma physics  literature  \cite{armstrong,gaje,holloway,klimas,koshka,parker, kormann2021,manzini,schumer,charli}, the theory of    Hermite basis  is  often motivated by  the development of plasma numerical  simulators.
The convergence of the Hermite-Fourier method with symmetric functions is provided in \cite{manzini,funaro2021,bourdiec}: unfortunately the theory is difficult to  extend to the 
asymmetric case.
The seminal reference is Holloway who discussed these two cases in \cite{holloway}.
When employing symmetrically-weighted basis functions, 
particle number, mass and total energy are preserved for odd number of moments, 
or momentum are preserved for even number of moments.  
On the other hand, asymmetrically-weighted basis functions, 
conserves particle number, total energy and momentum simultaneously.
In \cite{kormann2021}, K. Kormann and A. Yurova, 
using the idea of telescoping sums to show conservation 
for the symmetrically-weighted and asymmetrically-weighted cases, 
reached conclusions consistent with those in \cite{holloway}.
This is the reason why researchers are more interested 
in asymmetrically-weighted basis functions.

Subsequent research \cite{schumer} by Schumer and  Holloway 
conducted  comprehensive numerical simulations of both symmetric and asymmetric  methods, 
revealing that the asymmetrically-weighted Hermite method became numerically unstable. 
On the contrary, the symmetrically-weighted method 
is more  robust and suitable for long-time simulations.
Among recent  works which explicitly mention  the instability of asymmetric basis, we quote  \cite{manzini,bessemoulin} where the first contribution relies on adding a Fokker-Planck perturbation
to enforce stability
while  the second contribution details  a weighted norm which changes dynamically in time (its net effect is that the underlying reference temperature increases in time).
  Funaro and   Manzini provide in  \cite{funaro2021}  a mathematical investigation of  the stability of the Hermite-Fourier spectral approximation  of the Vlasov-Poisson model for a collisionless plasma in the electrostatic limit. The analysis includes high-order artificial collision operators of Lenard-Bernstein type.
In \cite{pagliantini2023}, the  authors proposed a spectral method 
for the 1D-1V Vlasov-Poisson system where the discretization in velocity space is based on asymmetrically-weighted Hermite functions, dynamically adapted through two velocity variables, which aims to maintain the stability of the numerical solution.
To our knowledge, none of the quoted works  has ever explained the origin of the numerical  instability of asymmetric basis. 

For the clarity of the presentation and because it corresponds to  the physical  scenarios we are interested in, we will illustrate the various instability and stability phenomenons in the context of oscillating
electric fields.  On the one hand   asymmetric Hermite functions 
have a kind of degeneracy for large values of the variable, see our Remark \ref{rem:2.1}.
Then it is clear that  an electric field with constant sign (in dimension one)  {\it pushes} the density of charge (positive or negative) towards high velocities $|v|\gg 1$ for which the approximation properties of  
 asymmetric Hermite functions deteriorate. Our understanding is that   asymmetric Hermite functions are not adapted to such scenarios.
 But on the other  hand the  oscillating fields,  which are naturally generated in plasma physics  thanks to global charge neutrality, are much less prone to  {\it push} the density of charge (positive or negative) towards high velocities $|v|\gg 1$,
 so the problem
 does not show up.
  In summary, physical applications we have in mind are more for {\it resonators} than for {\it accelerators} and that is why we only consider oscillating
electric fields in this work.

Our theoretical contributions are threefold.  Firstly we provide  a simple linear explanation of the numerical instability  phenomenon. Secondly 
we propose in Theorems \ref{theor:4.1} and \ref{theor:nm}
 original compact  formulas for the scalar product (\ref{eq:sp}) of two asymmetric Hermite functions.
Thirdly we 
 show how to modify the matrices of the discrete problem so as to recover the natural stability of the model problem, as in Lemma \ref{lem:maki} for example.
 
 The plan of this work is as follows. In Section \ref{sec:2}, we provide a simple example of the  instability attached to truncated asymmetric basis for the numerical simulation of $\partial_t f +e \partial_v f=0$.
 The electric field is reversed $e\rightarrow -e$  at a given time to emulate an oscillating electric field.
 Then in Section \ref{sec:3}, we analyze the structure of the Gram matrix of   asymmetric basis.
  Section \ref{sec:4} is devoted to the exact calculation of the coefficients of the Gram matrix, where we propose formulas which are new to our knowledge.
  In Section \ref{sec:5} we present two easy-to-implement modifications of the matrix of the problem, where the antisymmetry
with respect to the Gram matrix is recovered by construction.
The next Section \ref{sec:6} is devoted to a simple generalization to the equation $\partial_t f + v \partial_x f=0$.
Finally, Section \ref{sec:7}, we illustrate the general properties with numerical tests with oscillating electric fields. 

 The notations try to keep the technicalities to the minimum, and we will use the language of linear algebra
 to detail the properties of the various objects. 

\section{Notations and illustration of the numerical instability} \label{sec:2}

Our 
 model problem  is the  transport equation in  velocity 
 \begin{equation} \label{eq:b1}
\partial_t f(t,v) + e(t) \partial_ v f(t,v)=0.
\end{equation}
The function 
 $e(t)$ is a   space-constant electric field. For the reasons explained in the introduction, the sign of $e(t)$ changes (we will take $e(t)=\pm1$ in the numerical tests).
 Any solution of the equation preserves the quadratic norm
 \begin{equation} \label{eq:b1:ra}
 \frac d{dt} \int_\mathbb R f(t,v)^2dv=0.
 \end{equation}
 Let $(H_m)_{m\in \mathbb N}$ be the family of Hermite polynomials \cite{nist,szego,magnus} which is orthogonal with respect to the Gaussian weight $e^{-v^2}$.
 Introducing a reference temperature $T>0$ (for plasma physics applications), 
 the asymmetric basis $(\psi_m)_{m\in \mathbb N}$  and the asymmetric basis   $(\psi^m)_{m\in \mathbb N}$ are defined as 
 $$
 \left\{
 \begin{array}{lrl}
\psi_m(v)=& e^{\frac{-v^2}{T}} T^{-\frac12}  (2^m m!\sqrt{\pi})^{-\frac12}H_m(v/\sqrt T) , \\
\psi^m(v)=&  (2^m m!\sqrt{\pi})^{-\frac12} H_m(v/\sqrt T)
	\ &= e^{\frac {v^2}{T}}T^{\frac12} \psi_m(v).
\end{array}
\right.
$$
 Due to  orthogonality property $\int_\mathbb R \psi_m(v) \psi^n(v) dv=\delta_{mn}$, the two families are dual.
 The  classical  symmetric Hermite function corresponds to 
   $$
   \phi_m(v)=e^{\frac{v^2}{2T}}T^{\frac14}\psi_m(v)=e^{\frac{-v^2}{2T}}T^{-\frac14}\psi^m(v).
   $$
   The family of Hermite functions
 $(\phi_m)_{m\in \mathbb N}$ forms a complete orthonormal family (Hilbertian family) of $L^2(\mathbb R)$. 
Other important  identities   for all $ m\in \mathbb N $ are 
$$
(\psi_m)'(v)=-\sqrt{\frac{ 2(m+1)} T}\  \psi_{m+1}(v)  
\mbox{ and }
(\psi^m)'(v)=\sqrt{\frac {2m}T } \ \psi^{m-1}(v).
$$  
The second identity is natural because the derivative of a polynomial is a polynomial of lesser degree.
The first property can be deduced with the help of duality between  $(\psi_m)_{m\in \mathbb N}$   and   $(\psi^m)_{m\in \mathbb N}$. 

A  common procedure to discretize  (\ref{eq:b1}) 
starts from the a priori infinite representation 
\begin{equation} \label{nn:1}
f(v)=\sum_{m\geq 0} u_m \psi_m\left(v \right)
\end{equation}
where the coefficients $u_m$ are the moments of the function $f$. By definition one has  
\begin{equation} \label{eq:b2}
e^{\frac{v^2}{2T}}T^{\frac14}f(v)=\sum_{m\geq 0} u_m\phi_m\left(v \right).
\end{equation}
The   condition for the convergence  in $L^2(\mathbb R)$ of the series in  (\ref{eq:b2})  writes as 
\begin{equation} \label{eq:b3}
e^{\frac{v^2}{2T}}T^{\frac14} f \in L^2(\mathbb R)  \Longleftrightarrow 
\left\|e^{\frac{v^2}{2T}}T^{\frac14} f  \right\| _{L^2(\mathbb R)}^2= \sum_{m\geq 0} |u_m|^2<\infty  .
\end{equation}

\begin{Remark}
At inspection of (\ref{eq:b3}), it is clear that 
an arbitrary  translation $v\rightarrow v+a$ with large $a\gg 1$ of a given profile $f$ may dramatically   increase the quadratic norm $ \sum_{m\geq 0} |u_m|^2<\infty$ because of the term $e^{\frac{v^2}{2T}}$.
It explains why it is better to use asymmetric Hermite functions for the approximation of  profiles with small tails at infinity. In this context,  
an oscillating electric field offers   the possibility  to control the spread at infinity of the profile $f$.
\end{Remark}

Let $f(v,t)$ be a solution of the transport equation (\ref{eq:b1}). Under convenient convergence conditions on the series,  one has
$$
\left\{
\begin{array}{lll}
\partial_t f(t,v)=\sum_{m\geq 0} u_m'(t) \psi_m\left(v \right), \\
\partial_v f(t,v)=- \sum_{m\geq 0} u_m(t) \sqrt{\frac{2(m+1)}T}\psi_{m+1}\left(v \right).
\end{array}
\right.
$$
Therefore (\ref{eq:b1})  rewrites as
$$
\sum_{m\geq 0} u_m'(t) \psi_m\left(v \right) -e(t) \sum_{m\geq 0} u_m(t) \sqrt{\frac{2(m+1)}T}\psi_{m+1}\left(v \right)=0
$$
from which one deduces
the identities
\begin{equation} \label{eq:b5}
u_0'(t)=0 \mbox{ and }u_{m}'(t) -e(t)\sqrt{\frac{2m}T}u_{m-1}(t)  =0 \mbox{ for all }m\geq 1.
\end{equation}

\begin{figure}[h!]
\centering
\begin{tabular}{cc}
 \includegraphics[scale = 0.275]{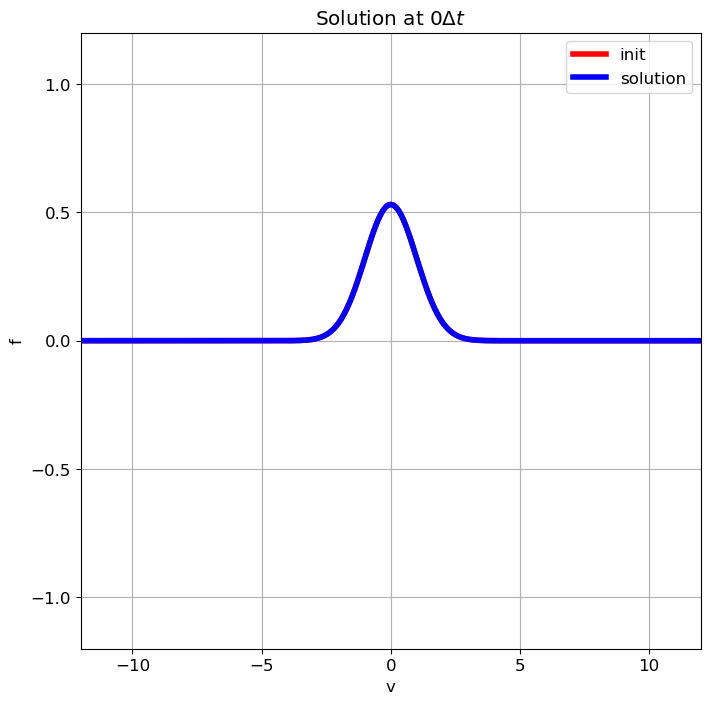}
&\includegraphics[scale = 0.275]{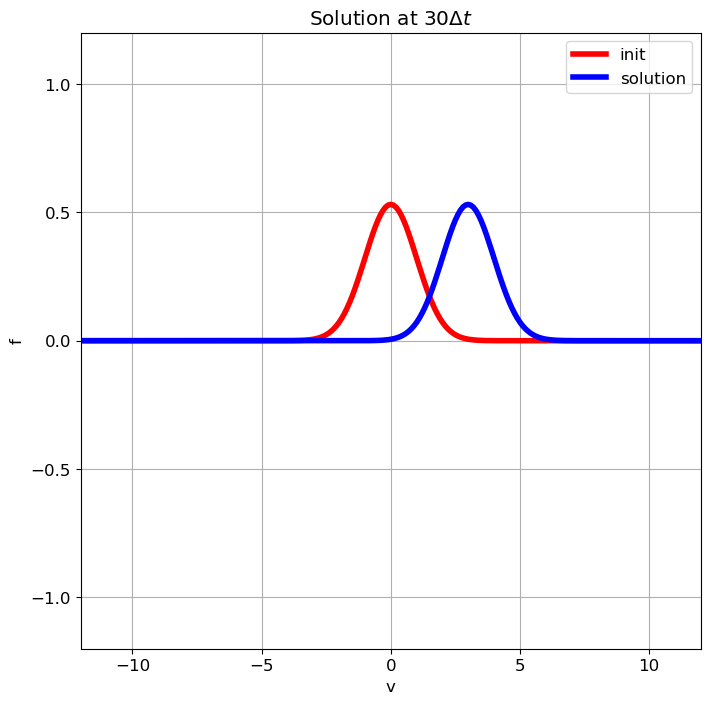} \\
 \includegraphics[scale = 0.275]{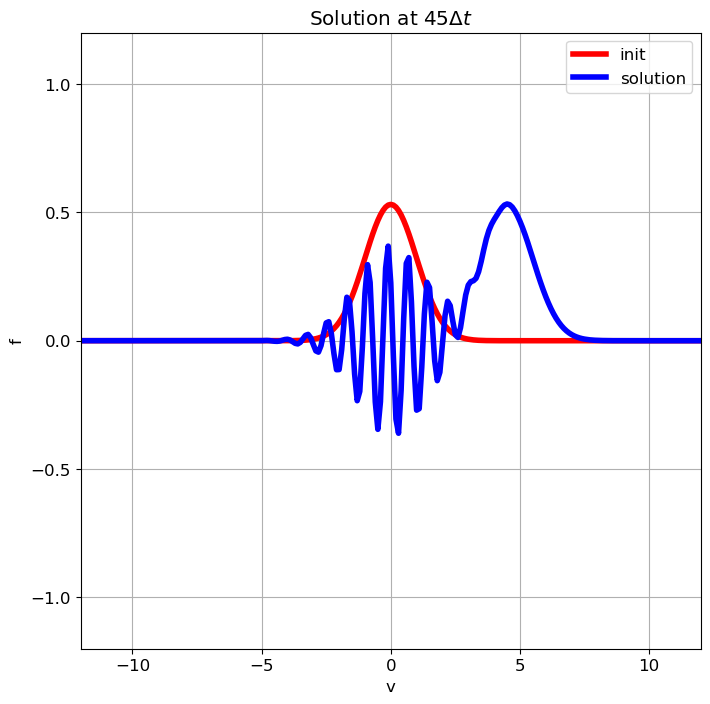}
&\includegraphics[scale = 0.275]{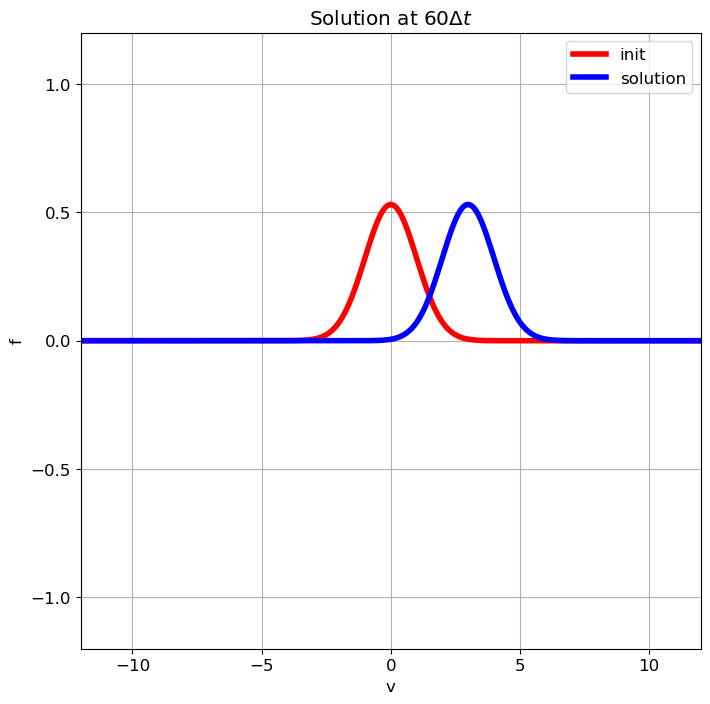} \\
 \includegraphics[scale = 0.275]{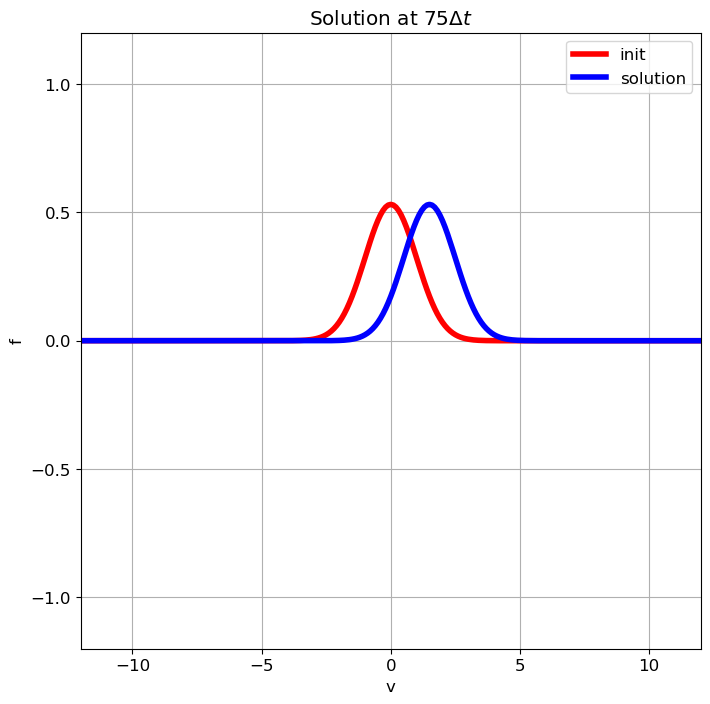}
&\includegraphics[scale = 0.275]{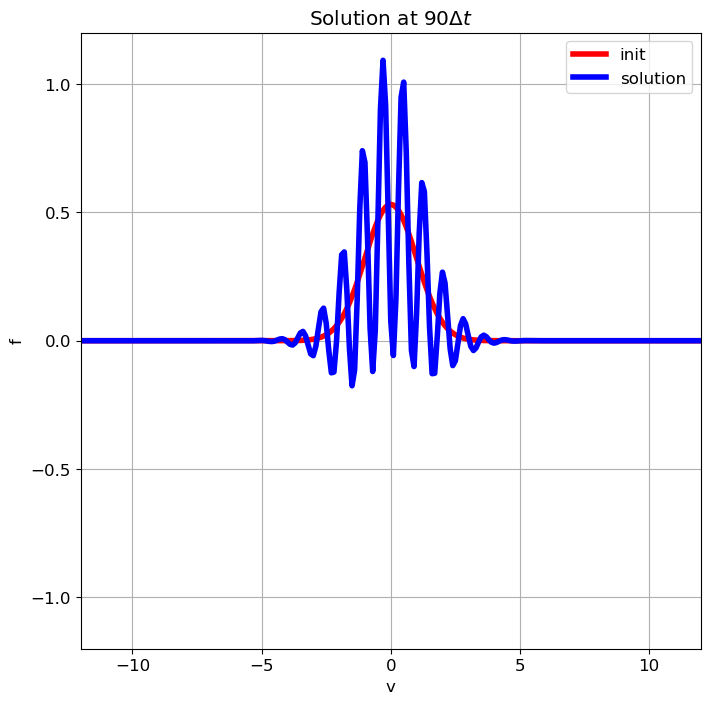} \\
 \end{tabular}
 \caption{Numerical exemple  of the advection test computed the scheme (\ref{eq:b8}) at time $t_0=0\Delta t$, $t_1=30\Delta t$, $t_2=45\Delta t$, $t_3=60\Delta t$, $t_4=75\Delta t$ then $t_5=90\Delta t$ ($N=64$ and $\Delta t=0.1$). Before time $t\approx t_2$,
 the solution is correct. The sign of the electric field is reversed at $t=4.5$. Then a numerical instability starts to be visible for  $t\approx t_2$, diminishes at later times, but is fully  unacceptable at the final time.}
 \label{fig:1}
\end{figure}

\begin{Remark} \label{rem:2.1}
The  equations (\ref{eq:b5})  display  two  remarkable   properties:
\begin{itemize}
\item  the variation of the first moment is zero which expresses that the density $\int f(t,v)v$
is constant in time, 
\item the other equations are ordered in an ascending series in the sense that the variation of $u_m$ depends only on $u_{m-1}$.
\end{itemize}
In our opinion, these properties are the reasons why the ascending series (\ref{eq:b5}) is very popular in plasma physics.
\end{Remark}
 
In this work, we will  systematically rewrite such relations as linear systems.
Let us define the infinite  triangular and sparse matrix $D\in \mathbb R^{\mathbb N\times \mathbb N}$ 
\begin{equation} \label{eq:b40}
D=(d_{mn })_{m,n\geq 0} , \qquad d_{mn}=- e \sqrt{\frac{ 2m} T}\delta_{m-1,n}.
\end{equation}
Only the first diagonal below the main diagonal is non zero.
The notation for the  infinite vector of moments is  $U(t) = (u_m(t))_{m\geq 0}\in \mathbb R^{\mathbb N}$. 
With these notations the transport  equation (\ref{eq:b1})   yields  the infinite system 
\begin{equation} \label{eq:b7}
\frac d{dt} U +  DU=0.
\end{equation}
Let $\Pi$ be the reconstruction operator such that
$$
\Pi U=f \mbox{ where }  U=(u_m)_{m\geq 0} \mbox{ and }f(v)=\sum_{m\geq 0} u_m \psi_m(v).
$$
This definition is formal in the sense that the spaces are not specified. From (\ref{eq:b7}) one can write
$
\Pi \partial_t U+ \Pi  D U=0
$. By construction on has $\Pi \partial_t U=\partial_t \Pi U=\partial_t f $ and 
\begin{equation} \label{eq:b25}
\Pi D  U= e\partial_v f.
\end{equation}
So (\ref{eq:b7}) formally implies the transport equation (\ref{eq:b1}).
To give a rigorous meaning to these formal   calculations, it is sufficient to take $U\in X$ where $X$ is the
space of  vectors with compact support
\begin{equation} \label{eq:xxx}
X=\left\{U\in \mathbb R^\mathbb N\mid u_m=0\mbox{ for } m \geq m_0 \right\}.
\end{equation}
The numerical discretization is easily performed with a simple truncation for moments between $0\leq m \leq N$. That is one considers the truncated 
vector of moments $U_N(t) = (u_m(t))_{0\leq m \leq N }\in \mathbb R^{N+1}$ and the truncated matrix 
$$
D^N=(d_{mn })_{0\leq m,n \leq N}, \quad N\in \mathbb R^{(N+1)\times (N+1)}.
$$
For the simplicity of the numerical analysis, the discretization will be systema\-tically performed with a Crank-Nicholson technique, which means that the fully discrete system writes
\begin{equation} \label{eq:b8}
\frac{U^{n+1}_N - U^n_N}{\Delta t}+  D^N \frac{U_N^n +U_N^{n+1}}{2}=0, \quad n\geq 0,
\end{equation}
where the time step is $\Delta  t>0$.
 This numerical method is in principle adapted to equations which preserve some quadratic energy as it is the case for the transport equation.
It has the advantage that, a priori,  no CFL condition is  required.

An example of a simulation is provided in Figure \ref{fig:1} at six different time 
$t_0=0\Delta t$, $t_1=30\Delta t$, $t_2=45\Delta t$, $t_3=60\Delta t$, $t_4=75\Delta t$, $t_5=90\Delta t$ and time step $\Delta t=0.1$.
The initial data is $U=(1,0, 0, \dots)$, that is only the first moment
is non zero. 
The electric field is
\[
e(t) = \left\{
\begin{aligned}
+1, \quad 0   \leq t <  4.5, \\
-1, \quad 4.5 \leq t \leq 9. \\
\end{aligned}
\right.
\]
The initial data is a pure Gaussian.
It is clear on the final result that the numerical simulation is spoiled with an important numerical instability which is  in clear contradiction with the preservation of the quadratic norm
(\ref{eq:b1:ra}). The numerical result is meaningless just after one oscillation of the electric field.
If one believes the numerical scheme (\ref{eq:b8}) is correct (which is the case),
then  the instability visible in Figure \ref{fig:1} is a paradox since the initial equation (\ref{eq:b1}) is stable.
The rest of this work is devoted to analyze the reason of this instability and to propose ways to control it.

\section{Structure of the Gram matrix} \label{sec:3}

To understand the nature of the problem at stake, 
let us expand the quadratic norm of $f(t)$ as
\begin{equation}\label{eq:quadratic_norm}
  \int_\mathbb R f(t,v)^2 dv=   \int_\mathbb R \left( \sum_{m\geq 0} u_m(t) \psi_m(v)\right) ^2 dv.
\end{equation}
Formally, that is considering that all sums are convergent, 
one has the double expansion
$$
 \int_\mathbb R \left( \sum_{m\geq 0} u_m(t) \psi_m(v)\right) ^2 dv= \sum_{m\geq 0} \sum_{n\geq 0} a_{mn} u_m(t)u_n(t)
$$
where the coefficients are 
\begin{equation} \label{eq:b20}
a_{mn}=a_{nm}= \int \psi_m\left(v \right)\psi_n\left(v \right)dv.
\end{equation}
It yields to the following definition which is central in our work.

\begin{defi}
The doubly infinite symmetric Gram matrix 
$
A=A^T=(a_{mn})_{m,n\geq 0}\in \mathbb R^{\mathbb N\times \mathbb N}$
of the problem is the collection of all    scalar products of the asymmetric basis functions.
\end{defi}
Let $\left< \cdot, \cdot \right>$ be the standard Euclidean scalar product between vectors.
By definition of the matrix $A$, one has formally 
\begin{equation} \label{eq:b26}
\left<A U, V \right>= \int_{\mathbb R} f(v)g(v) dv \mbox{ where } f=\Pi U \mbox{ and } g=\Pi V.
\end{equation}

\begin{Remark}
In the references \cite{bessemoulin,kormann2021}, the authors 
 consider a weighted scalar product 
$(\alpha, \beta)_\omega=\int_\mathbb R \alpha(v) \beta(v) \omega(v) dv$. 
 In both references the weight function is a  non trivial Maxwellian function
$\omega(v)=e^{v^2/T}$ where the parameter $T>0$ can take different values. 
The temperature $T=T(t)$ can  even change dynamically in time, as  in \cite{bessemoulin}.
In our case, the weight function is the trivial  one $\omega(v)\equiv 1$ and the scalar product  (\ref{eq:sp}) is  non weighted. 
\end{Remark}

The properties  of the Gram  matrix are studied below.

\begin{lemma} \label{lem:3.2}
The triangular matrix $D$ is skew-symmetric with respect to the scalar product induced by the matrix $A$, that is 
$
AD+D^TA=0 $. 
\end{lemma}

\begin{proof}
Take any $U,V\in X$.
One has
$$
\left< A D U, V\right>= \int_\mathbb R \Pi (DU)(v) g(v) dv=  \int_\mathbb R e \partial_v f(v)  g(v) dv.
$$
An integration by parts yields
$$
\left< A D U, V\right>=-  \int_\mathbb R f(v) e \partial_v f(v)   g(v) dv= - \left<U,  A D  V\right>.
$$
Since it holds for all $U,V\in X$, one gets the claim.
\end{proof}

\begin{Remark}
Based on this property,  a solution of (\ref{eq:b7}) satisfies the formal identities
\begin{equation} \label{eq:b29}
0=\left<A U, \frac d{dt}  U \right>+ \left< AU, DU\right>=\frac d{dt}  \frac{ \left< U,  A  U \right>} 2 + \left< U, ADU\right>=\frac d{dt} \frac{ \left< U,  A  U \right>} 2
\end{equation} 
since $AD$ is a skew-symmetric matrix. One recovers that the quadratic energy $\left< U,  A  U \right>$ is constant in time, see (\ref{eq:b1:ra}).
This  strongly suggests  that the instability visible in Figure \ref{fig:1} is a finite dimensional  effect caused by the truncation of the number of moments.
\end{Remark}

To analyze the effect of moment truncation on this phenomenon we decompose the lower triangular infinite matrix $D$ as
$$
D=\left(
\begin{array}{cc}
D_{11}^N & D_{12}^N \\
D_{21}^N & D_{22}^N
\end{array}
\right)
$$
where the blocks are 
$$
\left\{
\begin{array}{clcl}
D_{11}^N=D^N &\in \mathbb R^{(N+1)\times (N+1)} , &
D_{12}^N=0 &\in \mathbb R^{(N+1)\times \mathbb N} , \\
D_{21}^N&\in \mathbb R^{\mathbb N\times (N+1)} , &
D_{22}^N&\in \mathbb R^{\mathbb N\times \mathbb N} .
\end{array}
\right.
$$
Similarly  we decompose the infinite Gram matrix 
$$
A=\left(
\begin{array}{cc}
A_{11}^N & A_{12}^N \\
A_{21} ^N& A_{22}^N
\end{array}
\right)
$$
where the blocks are 
$$
\left\{
\begin{array}{clcl}
A_{11}^N &\in \mathbb R^{(N+1)\times (N+1)} , &
A_{12} ^N&\in \mathbb R^{(N+1)\times \mathbb N} , \\
A_{21}^N&\in \mathbb R^{\mathbb N\times (N+1)} , &
A_{22}^N&\in \mathbb R^{\mathbb N\times \mathbb N} .
\end{array}
\right.
$$

\begin{lemma} \label{Lem:3.5}
For all $N\geq 1$, one has 
\begin{equation} \label{eq:b30}
A_{11}^ND_{11}^N+ (D_{11}^N)^TA_{11}^N \neq 0.
\end{equation}
\end{lemma}

\begin{proof}
The equality $AD+D^TA=0$ reduces to
\begin{equation} \label{eq:fl}
\left\{
\begin{array}{ll}
A_{11}^ND_{11}^N+ A_{12}^N D_{21}^N + (D_{11}^N)^TA_{11}^N +(D_{21}^N)^TA_{21}^N =0, \\
A_{12}^N D_{22}^N+(D_{22}^N)^TA_{21}^N =0, \\
A_{21}^ND_{11}^N+ A_{22}^N D_{21}^N + (D_{11}^N)^TA_{21}^N +(D_{21}^N)^TA_{21}^N =0, \\
A_{22}^N D_{22}^N+ (D_{22}^N)^TA_{21}^N =0.
\end{array}
\right.
\end{equation}
One obtains $A_{11}^ND_{11}^N+  (D_{11}^N)^TA_{11}^N=-  A_{12}^N D_{21}^N-(D_{21}^N)^TA_{21}^N\in \mathbb R^{(N+1)\times (N+1)}$.
Since $D$ is an infinite  triangular matrix with only one non zero diagonal just below the main diagonal (see (\ref{eq:b40})), then the coefficients of 
$D_{21}^N$ are all zero except one at its top right corner which is non zero
\begin{equation} \label{eq:d12n}
 D_{21}^N= \left( 
\begin{array}{c|c|c|c|c|c|c|c|c|c}
0 & 0 & 0 & 0 & \dots & \dots & 
\begin{array}{c}
-e \sqrt{\frac{2(N+1)}{T}} \\
0 \\
0 \\
0 \\
\cdots \\
0 \\
0 \\
0 \\
\cdots
\end{array}
\end{array}
 \right)\in \mathbb R^{\mathbb N\times (N+1)}  .
\end{equation}
The multiplication by $A_{12}^N $ yields a matrix which is zero everywhere except its last column (which is proportional to the first column of  $A_{12}^N $) that is 
$$
A_{12}^N D_{21}^N= \left( 
\begin{array}{c|c|c|c|c|c|c}
0 & 0 & 0 & 0 & \dots & \dots & 
\begin{array}{c}
m_0 \\
m_1 \\
m_2 \\
m_3 \\
\cdots \\
\cdots \\
m_{N} 
\end{array}
\end{array}
 \right) \in \mathbb R^{(N+1)\times (N+1)}.
$$
One obtains $A_{12}^N D_{21}^N  - (D_{21}^N)^TA_{21}^N \neq 0$.
So one obtains  
$$
A_{11}^ND_{11}^N + (D_{11}^N)^TA_{11}^N =-  A_{12}^N D_{21}^N  - (D_{21}^N)^TA_{21}^N \neq 0
$$
which  yields  the claim. 
\end{proof}

In our opinion, this property (\ref{eq:b30}) expresses that fact that the truncation to a finite number of moments (parameter $N$, representing $N+1$ moments) 
spoils the skew-symmetric property explained 
in  Lemma \ref{lem:3.2}. 
That is why the boundedness (\ref{eq:b29})  of the solution is not preserved by moment truncation.


\section{Coefficients  of the Gram matrix} \label{sec:4}

The coefficients of the matrix $A$ are $L^2$  scalar products of asymmetric functions. These coefficients are computable in finite terms since the product
of two asymmetric functions can be expressed as a Gaussian function multiplied by a polynomial function.
However, to our knowledge,  the exact value of these coefficients is not available in the reference  literature on special functions \cite{nist,szego,magnus}.
For further developments in the next Section, we propose in this Section some formulas for the calculation of the quadratic scalar product of asymmetric Hermite functions.

\begin{Theorem} \label{theor:4.1}
If the sum of the indices is odd $m+n\in 2\mathbb N+1$, then $a_{mn}=0$. Otherwise  
\begin{equation} \label{eq:b46}
a_{m-l, m+l}=(-1)^l T^{-\frac12}2^{-2m - \frac{1}2}    \frac{(2m)!}{m! \sqrt{(m-l)!(m+l)!}}.
\end{equation}
\end{Theorem}

\begin{proof}
If $m+n$ is odd, then $\psi_m \psi_n$ is equal to a Gaussian function multiplied by an odd polynomial, so its integral vanishes. In this case $a_{mn}=0$.
So let us consider the other case.

One has 
$
\sqrt{ 2m/T } \  a_{mn}= \sqrt{ 2m/T } \ \int \psi_m(v) \psi_n(v)dv$. Using the general identity $
(\psi_m)'(v)=-\sqrt{\frac{ 2(m+1)} T}\  \psi_{m+1}(v)$, one can write
$$
\sqrt{ 2m/T } \  a_{mn}= -  \int \psi_{m-1}'(v) \psi_n(v)dv=  \int \psi_{m}(v) \psi_n'(v)dv
$$
$$
=
- \sqrt{ 2(n+1)/T }   \int \psi_{m}(v) \psi_{n+1}(v)dv = - \sqrt{ 2(n+1)/T }   \  a_{m-1,n+1}.
$$
That is
\begin{equation} \label{eq:b40-rab}
\sqrt{ m } \  a_{mn}=- \sqrt{ n+1 }   \  a_{m-1,n+1}.
\end{equation}
One gets by iteration
$$
{\left( m(m-1) \dots 2 \right) }^\frac12   a_{mn}= (-1)^m { \left( n(n+1) \dots (m+n)\right)}^\frac12 a_{0,m+n}
$$
 that is
\begin{equation} \label{eq:b41}
a_{mn}=(-1)^m \left( {\frac{(n+m)!}{n!m!}} \right)^\frac12a_{0,m+n}.
\end{equation}
The technical Lemma \ref{lem:4.2} yields the value of $a_{0,m+n}$ from which one obtains 
$$
a_{mn}=(-1)^m \left( {\frac{(n+m)!}{n!m!}} \right)^\frac12(-1)^{(m+n)/2} T^{-\frac12}2^{-(m+n) - \frac{1}2}  \frac{ (m+n)!^\frac12} { \left( \frac{m+n}2 \right)! }
$$
that is
\begin{equation} \label{eq:amn_explicit_final}
a_{mn}=(-1)^{\frac{m-n}2}T^{-\frac12}2^{-(m+n) - \frac{1}2} \frac{(m+n)!}{ \left( \frac{m+n}2 \right)! \sqrt{m!n!} }.
\end{equation}
This is the claim up to the change of indices $(m,n)\leftarrow (m-l, n+l)$.
\end{proof}

\begin{lemma} \label{lem:4.2}
Let $m\in 2\mathbb N$. 
One has
$
a_{0m}=(-1)^{m/2} T^{-\frac12}2^{-m - \frac{1}2}  \frac{ (m!)^\frac12} { (m/2)! }$. 
\end{lemma}
\begin{proof}
One has
$
\psi_0(v)\psi_m(v)=
T^{-1}\pi^{-\frac12} (2^mm!)^{-\frac12} e^{-2v^2/T}H_m(v/\sqrt{T})$. 
To be able to  perform a rescaling in this expression, 
 one can use  the general formula \cite[page 255]{magnus}
$$
H_m(\lambda x)=  \sum_{l=0}^{[m/2]} \lambda^{m-2l} (\lambda^{2}-1)^l  \frac{m!}{(m-2l))!l!}H_{m-2l}(x)
.
$$ 
Take $\lambda=1/\sqrt 2$ and $x= \sqrt 2 v/\sqrt{T}$. Then
$$
H_m(v/\sqrt{T})= \left( -  \frac{1}{2} \right)^{m/2}
\frac{m!}{(m/2)!} +R(v)
$$
where the residual $R(v)$ is orthogonal to the weight $e^{-2v^2/T}$ because it is a linear combination of Hermite polynomials of degree $\geq 1$ (with convenient weight).
One obtains
$$
a_{0m}=\int \psi_0(v)\psi_m(v) dv =T^{-1}\pi^{-\frac12} (2^mm!)^{-\frac12}  \left( -  \frac{1}{2} \right)^{m/2}  \frac{m!}{(m/2)!}  \sqrt{T\pi/2}
$$
which yields the   claim after simplification.
\end{proof}

\begin{lemma}
Take  $m\in \mathbb N$.
Then
\begin{equation} \label{eq:amm}
a_{mm}= 
T^{-\frac12}2^{-2m - \frac{1}2}  \frac{ (2m)!} { \left( m ! \right)^2 }.
\end{equation}
For large $m\gg 1$, one has 
$
a_{mm}\approx (T\pi 2 m) ^{-\frac12} $. 
\end{lemma}

This formula is approximately  in accordance with the fact that the amplitude of the Hermite functions decreases like $O(m^{-\frac14})$ in the 
main "support" of Hermite functions \cite{szego}.

\begin{proof}
The Stirling formula written as  $m!\approx \sqrt{2\pi m} (m/e)^m$ yields that
$$
a_{mm} \approx 
T^{-\frac12}2^{-2m - \frac{1}2} 
\frac{\sqrt{2\pi 2 m} (2m/e)^{2m}}{\left( \sqrt{2\pi m} (m/e)^m \right)^2}
\approx (T\pi 2 m) ^{-\frac12}  .
$$
\end{proof}

Unfortunately the previous  formula (\ref{eq:amm}) cannot be  used to calculate the coefficients in a stable manner because 
 the calculation on the computer of the   factorial  of  large natural numbers is difficult.
Nevertheless $a_{mm}=O(m^{-\frac14})$ which is an indication
that the ratio of large numbers in (\ref{eq:amm}) is asymptotically a  small number.
It gives the intuition of the following formulas which provide a stable method to calculate all coefficients.

\begin{Theorem} \label{theor:nm}
The coefficients of the Gram matrix can be evaluated with computationally stable formulas.
\\
i) To calculate the diagonal coefficients of the Gram matrix, use the recurrence formulas
\begin{equation} \label{eq:iter1}
\left\{
\begin{array}{ll}
a_{00}=(2T)^{-\frac12},\\
a_{m+1,m+1}=\frac{2m+1}{2m+2}a_{mm}, & m\geq 0.
\end{array}
\right.
\end{equation}
ii) To calculate the upper extra-diagonal coefficients of the  Gram matrix, use the recurrence formulas which starts from the diagonal
\begin{equation} \label{eq:iter2}
\begin{array}{ll}
\forall m\geq 1: \quad a_{m-l-1, m+l+1}= \sqrt{\frac{m-l}{m+l+1}} a_{m-l, m+l}, & l=0, \dots, m-1, \\
\end{array}
\end{equation}
iii) The lower diagonal coefficients are equal to the  upper extra-diagonal coefficients
\begin{equation} \label{eq:iter3}
\forall m\geq 1: \quad a_{m+l, m-l}=a_{m-l, m+l} \mbox{ for } 1 \leq l \leq  m.
\end{equation}
\end{Theorem}

\begin{proof}
The first set (\ref{eq:iter1}) of  formulas are deduced from (\ref{eq:amm}). The second set  (\ref{eq:iter2}) of formulas are deduced from (\ref{eq:amm}). 
The Gram matrix being symmetric, the symmetry  (\ref{eq:iter3}) is trivial.

The computational stability of the formulas is because the only operations are multiplication by positive  numbers $\leq 1$.
\end{proof}

\section{Application to truncated matrices} \label{sec:5}

Our  objective now is to modify the matrix $D^N=D_{11}^N$ such that one recovers the skew-symmetry with respect to $A^N=A_{11}^N$.
The new matrix will be denoted as
$ \overline D^N=\overline{ D}_{11}^N$ and one will typically enforce 
\begin{equation} \label{eq:maki}
A^N \overline D^{N} + (\overline D^{N} )^T A^N=0.
\end{equation}

\begin{lemma} \label{lem:maki}
Assume  the modified matrix satisfies  (\ref{eq:maki}).
Then the  solution of 
$$
\partial_ t U^N + \overline D^{N} U^N=0
$$
preserves the weighted quadratic norm, that is 
$
\frac d {dt} \left<U^N, A^N U^N \right> =0$. 
\end{lemma}

\begin{proof}
Indeed one has 
$$
\frac d {dt} \left<U^N, A^N U^N \right> = 2 \left< U^N, A^N  \partial_t U^N \right>= -  2 \left< U^N,  A_{11}^N\overline D^{N}   U^N \right>=0
$$
since the matrix $A^N \overline D^{N}$ is skew-symmetric (\ref{eq:maki}).
\end{proof}

However it is needed to modify $D^N$ as small as possible to keep the good approximation properties of this matrix.
We  will consider two methods.

\subsection{First method} \label{sec:5.1}

The first method is more efficient than the second one. However we discovered this possibility  through numerical  explorations
which explains why it is described more as a numerical recipe rather than the application of some general principle. 
The final result of the Section is described in Theorem \ref{prop:propi}.

Consider the formal identities of a solution of (\ref{eq:b7})
\begin{equation}\label{eq:formal_indentities_with_X}
\left<U, A \frac d{dt}  U \right>+ \left<U, ADU\right>=0
\end{equation} 
where the infinite vector has finite number of moments, that is $U \in X$ where $X$ is defined in (\ref{eq:xxx}). 
More precisely assume that $U=(U_1^N, U_2^N)^{T}$ with $U_2^N = \mathbf{0}$.
Substituting $U $ into 
\eqref{eq:formal_indentities_with_X}, one obtains
$$
  \left<U_1^N, A_{11}^N \frac d{dt}  U_1^N \right>
+ \left<U_1^N, \left(A_{11}^N  D_{11}^N +A_{12}^ND_{21}^{N}\right)U_1^N\right>
= 0.
$$ 
The first modified matrix $ \overline D^N=\overline{ D}_{11}^N$ is defined such that the previous identity is a triviality for solutions of 
$\partial_ t U_1^N + \overline D_{11}^{N} U_1^N=0$.
One obtains 
$
A_{11}^N \frac d{dt}  U_1^N +\left(A_{11}^N  D_{11}^N +A_{12}^ND_{21}^{N}\right)U_1^N=0$ which yields
$$
- A_{11}^N \overline D_{11}^{N}  U_1^N +\left(A_{11}^N  D_{11}^N +A_{12}^ND_{21}^{N}\right)U_1^N=0.
$$
Requiring that it holds for all possible $U_1^N\in \mathbb R^{N+1}$ yields 
\begin{equation} \label{eq:rui1}
{\overline D}_{11}^N =D_{11}^N +\left(A_{11}^N\right)^{-1} A_{12}^ND_{21}^{N},
\end{equation}
which leads to  formal identity
$  \left<U_1^N, \frac d{dt}  U_1^N \right>
+ \left<U_1^N, {\overline D}_{11}^N U_1^N\right>
= 0$.
In fact, this formal identities 
is equivalent to the formal identities \eqref{eq:formal_indentities_with_X}, 
considering $U$ in $X$ (the space of vectors with compact support). 
This can also be interpreted as the consideration of the quadratic norm 
of \eqref{eq:quadratic_norm} being taken into account. 
The following result shows that the requirement of skew symmetry is recovered.

\begin{lemma}
The modified matrix (\ref{eq:rui1}) is skew-symmetric with respect to the Gram matrix (\ref{eq:maki}).
\end{lemma}

\begin{proof}
By definition one has
$
A_{11}^N {\overline D}_{11}^N =A_{11}^N D_{11}^N +A_{12}^ND_{21}^{N}$. 
The right hand side is skew-symmetric because of the first line of (\ref{eq:fl}). Therefore one has $A_{11}^N {\overline D}_{11}^N+ \left( A_{11}^N {\overline D}_{11}^N \right)^T=0$.
\end{proof}

It remains to calculate explicitly 
the correction term 
   $\left(A_{11}^N\right)^{-1}A_{12}^ND_{21}^{N}$ for  the modified  matrix (\ref{eq:rui1}) to be  completely constructed.
   This is the purpose of the next technical results.
   For the first lemma, we consider  a matrix $X^N\in   \mathbb R^{(N+1)\times \mathbb N}$ 
 where  only the first column of the square matrix $X^N$ is non zero.
 \begin{equation} \label{eq:sxn}
X^N= \left(
\begin{array}{c|c|c|c| c}
z^N & 
0 & 0 & 0 &  \dots 
\end{array}
\right), \qquad \mbox{ where }0 , z^N\in \mathbb R^{N+1}.
\end{equation}

   \begin{lemma}
   Take  $z$ as the unique solution of $A_{11}^N z^N=g^N$ where $g^N\in \mathbb R^{N+1}$ is the first column of $A_{12}^N$.
Then one has 
$
{\overline D}^N 
= D_{11}^N + X^ND_{21}^{N}.
$
\end{lemma}

\begin{proof}
Let us study the equation 
$X^N D_{21}^N - \left(A_{11}^N\right)^{-1}A_{12}^ND_{21}^{N}=0$.
It  is equivalent to the equation
$$
\left( A_{11}^N X^N - A_{12}^N \right)D_{21}^N= 0 
\in  \mathbb R^{(N+1)\times (N+1)}.
$$
Due to the special form (\ref{eq:d12n})  of $D_{21}^{N}$, the first $N$ columns of $\left( A_{11}^N X^N - A_{12}^N \right)D_{21}^N$  vanish identically.
Moreover the last column also vanishes provided the first column of $A_{11}^N X^N - A_{12}^N $ vanishes as well.
It writes $A_{11}^N z^N - g^N =0$
which corresponds to the claim.
\end{proof}

If  $N=2M$ is even, then one can check that all coefficients  with an even index  of  $g^N$ vanish by construction.
On the other hand if  $N=2M+1$ is odd, then one can check that all coefficients  with an odd index of  $g^N$  vanish by construction.
Since one diagonal over two consecutive ones of  matrix $A_{11}^N$ vanish, this is also the case for the inverse matrix $\left( A_{11}^N\right)^{-1}$.
It explains that the vector $z^N$  has an additional structure:
\begin{equation} \label{eq:szn}
\begin{array}{llll}
\mbox{for  even  } N=2M,   &  z_{2k}^N=0 \mbox{ for all }0\leq k \leq M, \\
\mbox{for  odd  } N=2M+1, &  z_{2k+1}^N=0 \mbox{ for all }
0\leq k \leq M.
\end{array}
\end{equation}
The index of lines is counted from 0 because it is in accordance with the index of the first
moment which is 0 as well.
We split the calculation of the coefficients of the vector  in two cases depending on the parity of $N$.
   
   \subsubsection{First case $N=2M$, $M\in\mathbb{N}^+$}

\begin{proposition}\label{prop:5.3}
$
z_{2k+1}^N = \sqrt{\dfrac{(2M+1)!}{(2k+1)!}} \left(\dfrac{-1}{4^{M-k}(M-k)!}\right)$ for all $0\leq k < M$.
\end{proposition}
\begin{proof}
The proof goes by checking the identity  $A_{11}^Nz^N =g^N$ where the coefficients of $z^N$ with odd index are given in the claim. The coefficients with even index vanish identically.

According to \eqref{eq:amn_explicit_final} in Theorem \ref{theor:4.1},  one has for $0 < p=2q+1 < 2M$ 
$$
a_{p, 2k+1} 
= (-1)^{\frac{p-(2k+1)}{2}} 
  \, T^{-\frac12}
  \, 2^{-(p+2k+1) - \frac{1}2} 
  \, \frac{(p+2k+1)!}{ (\frac{p+2k+1}{2})! \sqrt{p!(2k+1)!}}.
$$
Here  $p$ is the index of a line of the matrix $A_{11}^N$.
Using the value given in the claim, one calculates 
$$
\begin{aligned}
\sum_{k=0}^{M-1} a_{p, 2k+1} z_{2k+1}^N
=&(-1)^{\frac{p-(2M+1)}{2}} 
  \, T^{-\frac12}
  \, 2^{-(p+2M+1) - \frac{1}2} \\
 & \times \sum_{k=0}^{M} (-1)^{M-k+1} \dbinom{p+2k+1}{p} 
  \dbinom{(p+2M+1)/2}{(p+2k+1)/2} \\
 &\times \frac{1}{\sqrt{p!(2M+1)!}}\frac{p!(2M+1)!}{ (\frac{p+2M+1}{2})! }.\\
\end{aligned}
$$
Then using Lemma \ref{lem:5.1}, one has 
$$
\begin{aligned}
\sum_{k=0}^{M-1} a_{p, 2k+1} z_{2k+1}^N
=&(-1)^{\frac{p-(2M+1)}{2}} 
  \, T^{-\frac12}
  \, 2^{-(p+2M+1) - \frac{1}2} \\
 &\times \dbinom{p+2M+1}{p} 
  \frac{1}{\sqrt{p!(2M+1)!}}\frac{p!(2M+1)!}{ (\frac{p+2M+1}{2})! }.\\
=&(-1)^{\frac{p-(2M+1)}{2}} 
  \, T^{-\frac12}
  \, 2^{-(p+2M+1) - \frac{1}2} 
  \frac{(p+2M+1)!}{ (\frac{p+2M+1}{2})! \sqrt{p!(2M+1)!}}\\
=&a_{p, 2M+1}.
\end{aligned}
$$
That is $\sum_{k=0}^{M-1} a_{p, 2k+1} z_{2k+1}^N=a_{p, 2M+1}$ for all $0< p< 2M$. 
Note that the coefficients $a_{p, 2M+1}$ for all $0< p < 2M$ are those of $g^N$ (one coefficient over two consecutive ones).
 Therefore
the coefficients given in the claim are exactly the coefficients of $z^N$. 
\end{proof}

To show the following purely technical Lemma \ref{lem:5.1} needed in the above proof, we define
$$
S(N,p)
= \sum_{k=0}^{M-1} (-1)^{M-k+1} \dbinom{p+2k+1}{p} 
  \dbinom{(p+2M+1)/2}{(p+2k+1)/2}
  $$
  where $ N=2M$ is even and   $p=2q+1$ (for all $0\leq q  <M$) is odd.

\begin{lemma}\label{lem:5.1}
$
S(N,p)
= \dbinom{p+2M+1}{p} 
$.
\end{lemma}
\begin{proof}
One checks the identity
$$
S(N,p)
=\sum_{k=0}^{M} (-1)^{M-k+1} \dbinom{p+2k+1}{p} 
  \dbinom{(p+2M+1)/2}{(p+2k+1)/2}
+ \dbinom{p+2M+1}{p}
$$
rewritten as
\begin{equation} \label{eq:snp}
S(N,p)
=\sum_{k=0}^{M} (-1)^{k} {P}(N,p,k) \dbinom{M}{k} 
+ \dbinom{p+2M+1}{p}
\end{equation}
where $P(N,p,k)$ is defined as
$$
\begin{array}{rrccc}
{P}(N,p,k)
=&(-1)^{M+1}  \dfrac{ ((p+2M+1)/2)! }{p!M!} &\times & 
  \dfrac{ (p+2k+1)!   k!}{ (2k+1)!  \left( (p+2k+1)/2 \right)!} \\
  =& C(N,p)& \times & \dfrac{ (2q+2k+2)! k!  }{ (2k+1)!  \left( q+k+1 \right)!}.
  \end{array}
$$
By direct expansion, one checks  $P(N,p,k)$ 
can be written as  a polynomial with respect to the variable $k$. 
To show this fact,  define $A(k,q)= \dfrac{ (2q+2k+2)! k!  }{ (2k+1)!  \left( q+k+1 \right)!}$.
It is clear that $A(k,0)=2$. It is also clear that
$$
A(k,q+1)= \frac{(2q+2k+4)(2q+2k+3)}{q+k+2}  A(k,q)= 2(2q+2k+3) A(k,q).
$$
By iteration, one has that $A(k,q)$ is a polynomial in $k$ of degree $q$. So ${P}(N,p,k)$ is also a polynomial in $k$ of degree $q=(p-1)/2<M$.

On the other hand, one has the general identity for all degrees $r <M$.
$$
\sum_{k=0}^{M} (-1)^{k} k^r \dbinom{M}{k} = 0,
\quad r < M.
$$

Since ${P}(N,p,k)$ is  a polynomial in $k$ of the convenient degree, then the sum in (\ref{eq:snp}) vanishes, which ends the proof.
\end{proof}

\subsubsection{Second case $N=2M-1$, $M\in\mathbb{N}^+$}   

The analysis is very similar to the first case.

\begin{proposition}\label{prop:5.4}
$
z_{2k}^N = \sqrt{\dfrac{(2M)!}{(2k)!}} \left(\dfrac{-1}{4^{M-k}(M-k)!}\right)
$ for all $0\leq k < M$.
\end{proposition}

\begin{proof}
The proof is similar to that of Proposition \ref{prop:5.3}. It relies on the technical Lemma \ref{lem:techi2}.
\end{proof}

Define
$$
S(N,p)
= \sum_{k=0}^{M-1} (-1)^{M-k+1} \dbinom{p+2k}{p} 
  \dbinom{(p+2M)/2}{(p+2k)/2}
$$
where $N=2M-1$ is odd
and $p=2q$ (for all $0\leq q <M$) is even.

\begin{lemma} \label{lem:techi2}
$S(N,p)
= \dbinom{p+2M}{p} 
$.
\end{lemma}

\begin{proof}
The proof is similar to that of Lemma \ref{lem:5.1}, so it is omitted.
\end{proof}

\subsubsection{Final result}

\begin{Theorem} \label{prop:propi}
The modified matrix (\ref{eq:rui1}) can be written as  
$$
{\overline D}^N 
= D_{11}^N -e \sqrt{\frac{2(N+1)}{T}}
Y^N
$$
where the last matrix is 
$
Y^N= \left(
\begin{array}{c|c|c|c| c}
0 & 
0 & \dots  & 0 & z^N 
\end{array}
\right)\in \mathbb R^{(N+1)\times (N+1)}$ and $ 0 , z^N \in \mathbb R^{N+1}$.

For $N=2M$ even, the first (equation for mass) and third line (equation for kinetic energy) vanish.
For $N=2M+1$ odd, the second line (impulse) vanish.
\end{Theorem}

\begin{proof}
It comes from the structure (\ref{eq:d12n}) of $D_{12}^N$ and the structure (\ref{eq:sxn}) of $X^N$.
The coefficients vanish accordingly to (\ref{eq:szn}).
\end{proof}

In addition, according to Proposition \ref{prop:5.3} and \ref{prop:5.4}, 
we have
$$
z_{1} = \sqrt{(2M+1)!} \left(\dfrac{-1}{4^{M}(M)!}\right).
$$
and 
$$
z_{0} = \sqrt{(2M)!} \left(\dfrac{-1}{4^{M}(M)!}\right), \qquad 
z_{2} = \sqrt{\dfrac{(2M)!}{(2)!}} \left(\dfrac{-1}{4^{M-1}(M-1)!}\right).
$$
When $M$ takes a relatively large value, 
these corresponding values can become notably small, 
thereby allowing (up to arbitrary precision of course) for simultaneous conservation of mass, momentum, and energy.
For example, with $M=20$, we observe $z_0 = 3.38 \times 10^{-7}$, $z_2 = 1.9102 \times 10^{-5}$ and $z_1 = 2.162 \times 10^{-6}$.

\begin{Remark} \label{rem:remi}
In practice one can as well neglect very small coefficients. For example one can nullify $z_1$ or the pair $z_0,z_2$.
This will be used in one of the numerical tests.
\end{Remark}

\subsection{Second method} \label{sec:penal}

The second method uses a natural penalization technique with a parameter $\varepsilon>0$ and is much simpler.
The modified matrix is now defined as 
\begin{equation} \label{eq:maki2}
\begin{array}{lll}
  \overline D^N &=&  D^N -\frac12 (A^N+\varepsilon I^N)^{-1}\left( (A^N+\varepsilon I^N) D^N+ (D^N)^T (A^N+\varepsilon I^N)   \right) \\
  & = & \frac12 D^N -\frac12 (A^N+\varepsilon I^N)^{-1}  (D^N)^T (A^N+\varepsilon I^N).
  \end{array}
\end{equation}
The penalization term $\varepsilon I^N$ helps to calculate the inverse matrix $(A^N+\varepsilon I^N)^{-1}$ because we have observed that  the condition number of the matrix $A^N$ blows up as $N$ increases.
The drawback of the second method with respect to the first one is that it incorporates an additional source of approximation through the penalization parameter.

\begin{lemma} \label{lem:5.8}
Assume (\ref{eq:maki2}).
Then the  solution of 
$$
\partial_ t U^N + \overline D^{N} U^N=0
$$
preserves the weighted quadratic norm with penalization, that is 
$$
\frac d {dt}\left(  \left<U^N, A^N U^N \right> + \varepsilon \| U^N\|^2\right) =0.
$$
\end{lemma}
\begin{proof}
Similar as proof of Lemma \ref{lem:maki}.
\end{proof}

%
%
%
%
%
%
%
%
%

\section{Generalization to $\partial_t f + v \partial_x f=0$} \label{sec:6}

Consider in 1D the equation 
\begin{equation} \label{eq:mlo}
\partial_t f (t,x,v)+ v\partial_x f(t,x,v)=0.
\end{equation}
This equation is usually  a building block  in a splitting strategy used for the numerical discretization of a classical Vlasov equation such as
\begin{equation} \label{eq:vlasov-poisson}
\partial_t f +\mathbf v \cdot \nabla_x f + (\mathbf E+\mathbf v \times \mathbf B)\cdot \nabla_v f=0.
\end{equation}
The other building block is the transport equation (\ref{eq:sp}).

Discretization of the model problem  (\ref{eq:mlo}) with the method of moment  yields the infinite differential system 
\begin{equation} \label{eq:muki1}
\partial_t U+ B\partial_x U=0
\end{equation}
where $B=(b_{nm})\in \mathbb R^{\mathbb N\times \mathbb N}$ is defined by its coefficients 
\begin{equation} \label{eq:bnm}
b_{nm}=T^\frac12 \left( \sqrt{\frac{m+1}2}\delta_{m+1,n}+ \sqrt{\frac{m}2}\delta_{m-1,n} \right).
\end{equation}

\begin{lemma} \label{lem:anorm}
Formal solutions to (\ref{eq:muki1}) satisfy  the conservation of two quadratic norms
 $\frac d{dt}\int_x  \| U\|^2=0$ and $
\frac d{dt}\int _x \left< U, AU \right>=0 $.
\end{lemma}
\begin{proof}
Evident.
\end{proof}

\begin{Corollary}
The matrix $B$ is symmetric  $B^T=B$ and is symmetric with respect to the Gram matrix $AB=BA$. 
\end{Corollary}

\begin{proof}
The first property is just the definition (\ref{eq:bnm}). 
The second is a corollary of the preservation of the  norm $ \left< U, AU \right>$ of Lemma \ref{lem:anorm}: a detailed proof can be performed with the method of Lemma \ref{lem:3.2}.
\end{proof}

The approximation of $B$ with a finite
number of moments yields the matrix $B^N:=B_{11}^N\in \mathbb R^{(N+1)\times (N+1)}$ where
$$
B=\left(
\begin{array}{cc}
B_{11}^N & B_{12}^N \\
B_{21}^N & B_{22}^N
\end{array}
 \right)
$$ 
with $B_{11}^N\in \mathbb R^{(N+1)\times (N+1)}$, $B_{12}^N =(B_{21}^N )^T \in \mathbb R^{\mathbb N \times (N+1)} $ and $B_{22}^N \in \mathbb R^{\mathbb N\times \mathbb N}$.

Since $B^N$ is symmetric by construction, then 
the solution of the equation $\partial_t U^N + B^N \partial_x U^N=0$ preserves the quadratic norm $ \frac d{dt}\int_x  \| U^N\|^2=0$.
However one important  problem remains, which is the fact that 
the equation $\partial_t f + v \partial_x f=0$   is usually just one stage in a splitting algorithm where many basic equations are discretized.
The other equation can be for instance the model equation $\partial_t f + e \partial_v f=0$ (actually this block is always present
in all physical models we are interested in).
That is why it is important to guarantee the stability with respect to a criterion which is common to all parts of the general method.
This criterion is the preservation of the norm $ \left< U, AU \right>$.

In our opinion, the only way to obtain such a general criterion is to modify $B^N$ as it was done for $D^N$  so that it becomes symmetric with respect to $A^N=A_{11}^N$.
The two methods  for the modification of the matrix $D^N$ are quite easy to generalize to the matrix $B^N$ so we provide only the main ideas.

A first modified matrix writes
\begin{equation} \label{eq:bmod1}
\overline B^N= B^N+ (A_{11}^N)^{-1} A_{12}^N B_{21}^N.
\end{equation}

\begin{lemma}
The modified matrix (\ref{eq:bmod1}) is symmetric with respect to the truncated Gram matrix $A^NB^N= B^N A^N$ 
and is computable explicitly. 
\end{lemma}

\begin{proof}
The first property is because 
$
A_{11}^N \overline B^N= A_{11}^N B_{11}^N+  A_{12}^N B_{21}^N
$ is a symmetric matrix since $B$ is symmetric.
The second property is because $B_{21}^N$ is proportional to $D_{21}^N$. 
More precisely one has the relation   $B_{21}^N= -\frac{T}{2e}D_{21}^N$ which is deduced from (\ref{eq:d12n}) and (\ref{eq:bnm}).
So one obtains
$$
{\overline B}^N 
= B_{11}^N +  \sqrt{\frac{T(N+1)}{2}}
Y^N
$$ 
where the matrix $Y^N$ is explicitly given in Proposition \ref{prop:propi}.
\end{proof}

Another  possibility 
is to use to penalization technique of Section \ref{sec:penal}. That is we define
the second modified matrix  (still with the notation $A^N=A_{11}^N$)
\begin{equation} \label{eq:muki2}
\begin{array}{lll}
\overline B^N&=& B^N -\frac12 (A^N+\varepsilon I^N)^{-1} \left( (A^N+\varepsilon I^N)  B^N - B ^N (A^N+\varepsilon I^N) \right) \\
& = & \frac12  B^N + \frac12 (A^N+\varepsilon I^N)^{-1} B^N (A^N+\varepsilon I^N).
\end{array}
\end{equation}

\begin{lemma}
Solutions to $\partial_t U^N + \overline B^N \partial_x U^N=0$ with the first (resp. second)  modified matrix preserves the  norm $
   \left<U^N, A^N U^N \right>  $
 (resp. $ \left<U^N, A^N U^N \right> +\varepsilon \| U^N \|$).
\end{lemma}

\begin{proof}
Evident.
\end{proof}

\section{Numerical illustrations} \label{sec:7}

We have implemented a specialized research code\footnote{Repository: \url{https://gitlab.lpma.math.upmc.fr/asym}} in Python
to evaluate the new moments methods. 
We discretize space with a finite difference (FD) method and time with a Crank-Nicolson scheme.
Solving each time step involves solving a set of linear equations using the Krylov method, specifically the GMRES\cite{youcef} method, with the initial guess derived from the solution at the preceding time step.

\subsection{Transport equation}

In this Section, we recalculate the test of Figure \ref{fig:1} with $N=64$, $\Delta t=0.1$ and $T=2$
where the electric field is reversed at $t=4.5$.
We use the stabilized method explained in Section \ref{sec:5}.
In Figure \ref{fig:2} we plot the results where the matrix $\overline D^N$ is obtained with the method of Section \ref{sec:5.1}.
\begin{figure}[h!]
\centering
\begin{tabular}{cc}
 \includegraphics[scale = 0.275]{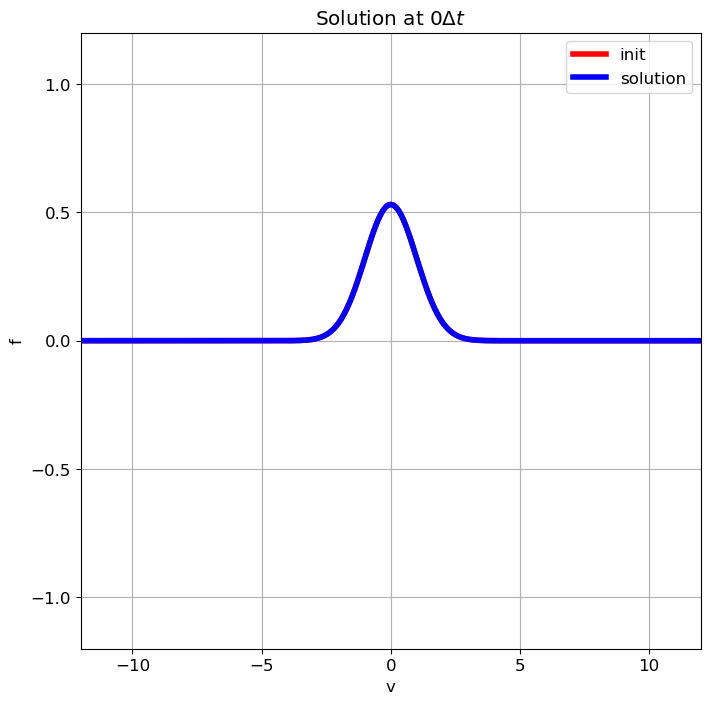}
&\includegraphics[scale = 0.275]{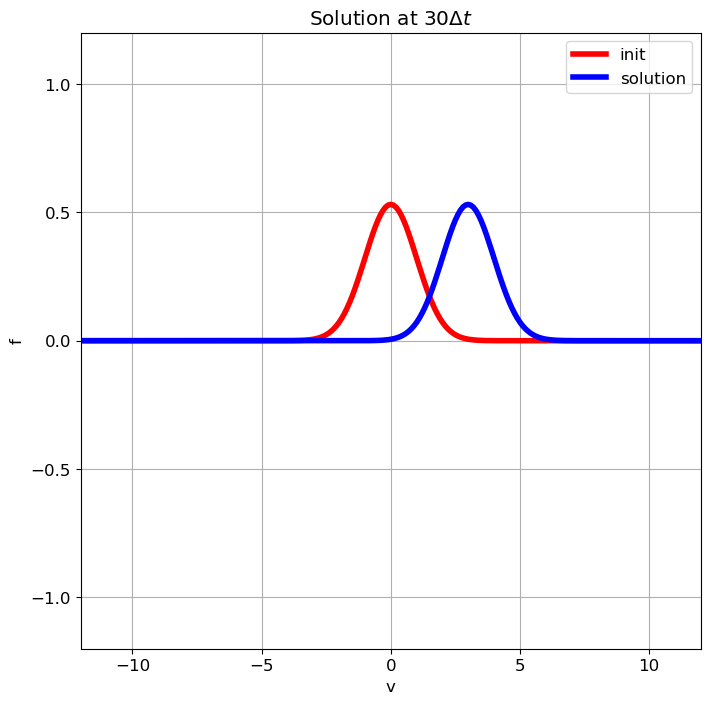} \\
 \includegraphics[scale = 0.275]{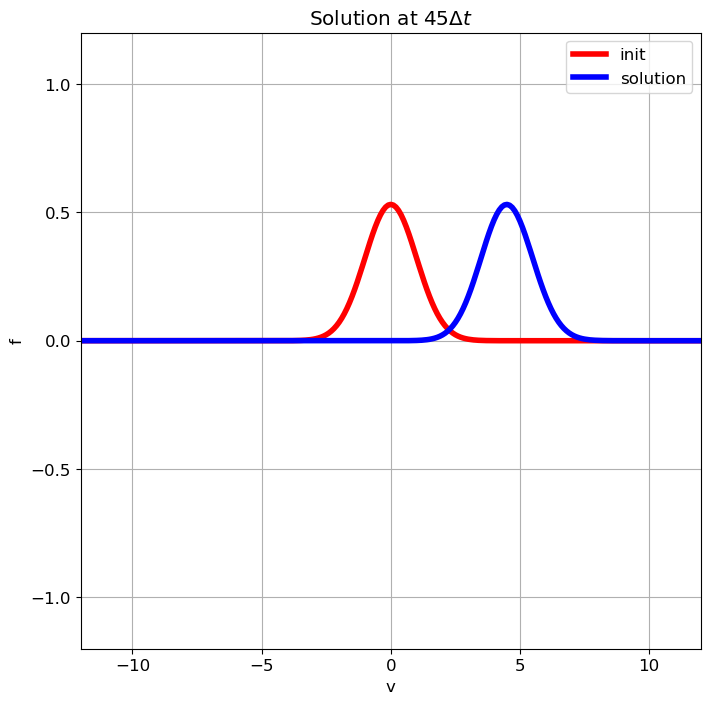}
&\includegraphics[scale = 0.275]{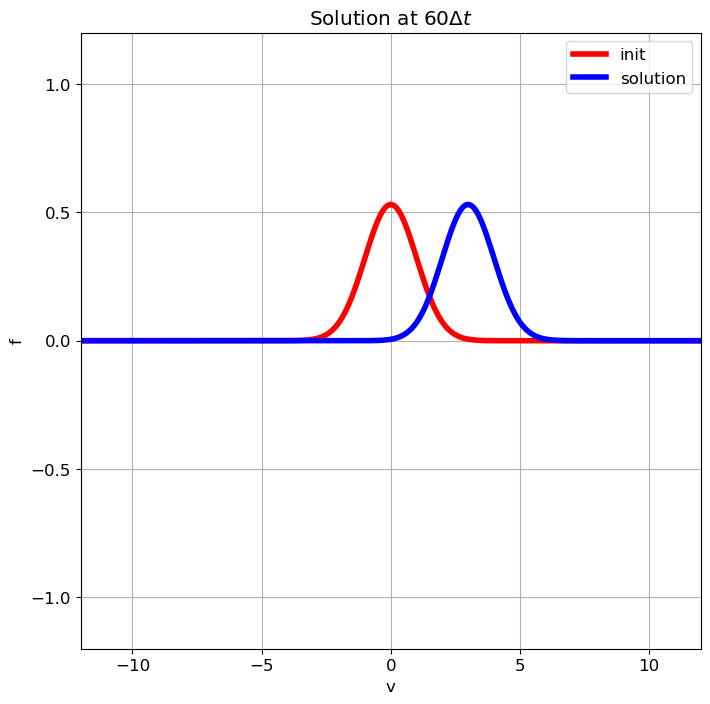} \\
 \includegraphics[scale = 0.275]{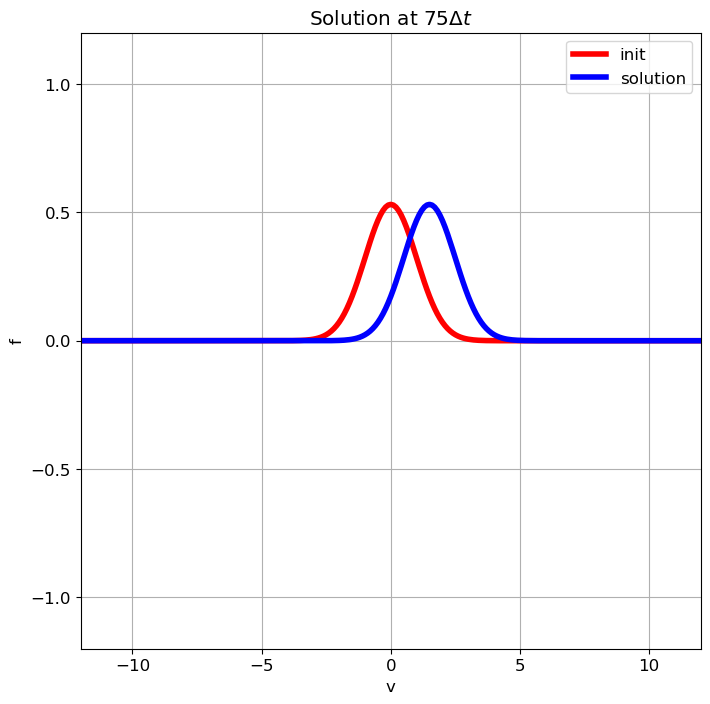}
&\includegraphics[scale = 0.275]{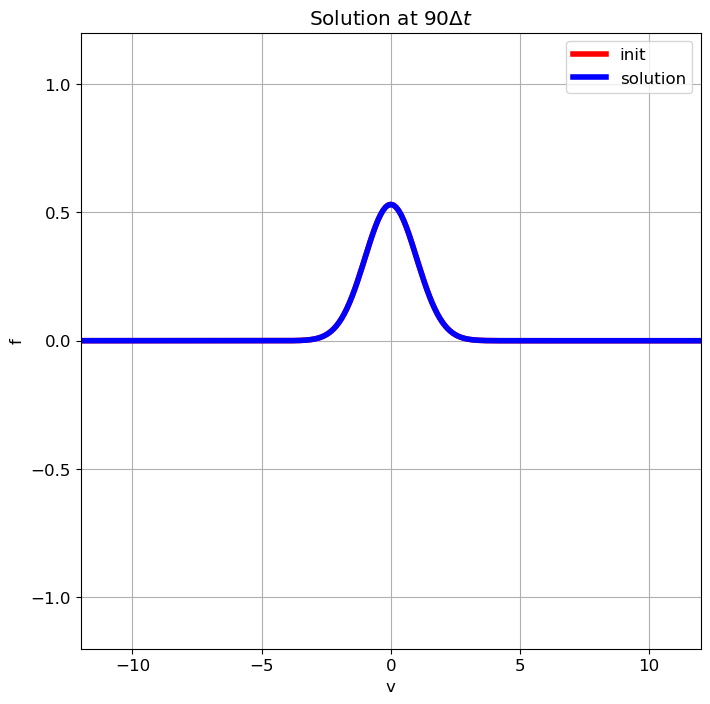}
 \end{tabular}
 \caption{Results of the advection test computed the scheme (\ref{eq:b8}) at time $t_0=0\Delta t$ to $t_5=90\Delta t$ ($N=64$ and $\Delta t=0.1$) with the first method.}
 \label{fig:2}
\end{figure}
\begin{figure}[h!]
\centering
\begin{tabular}{cc}
 \includegraphics[scale = 0.275]{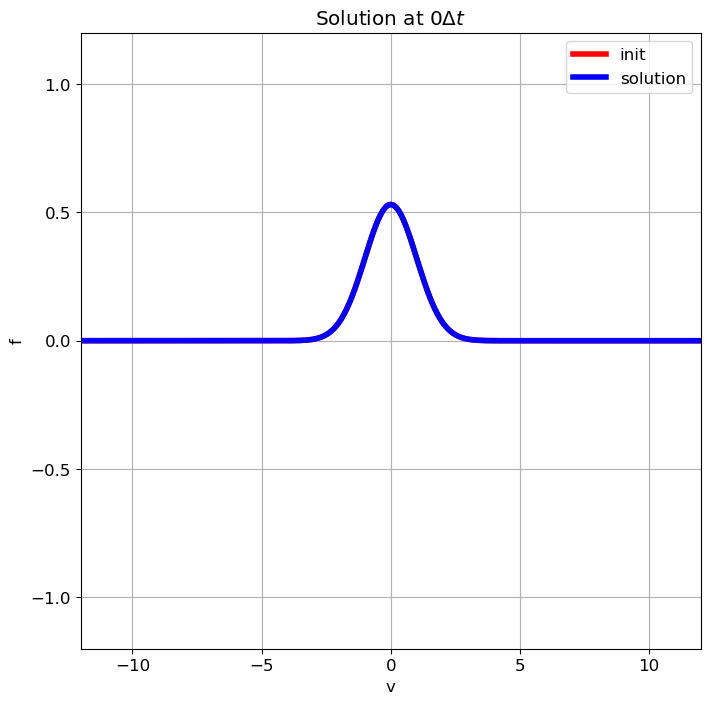}
&\includegraphics[scale = 0.275]{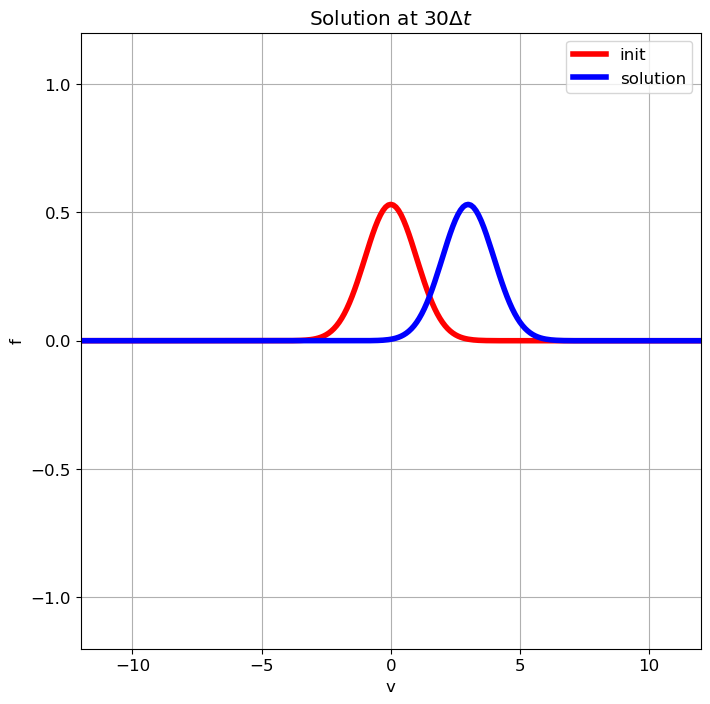} \\
 \includegraphics[scale = 0.275]{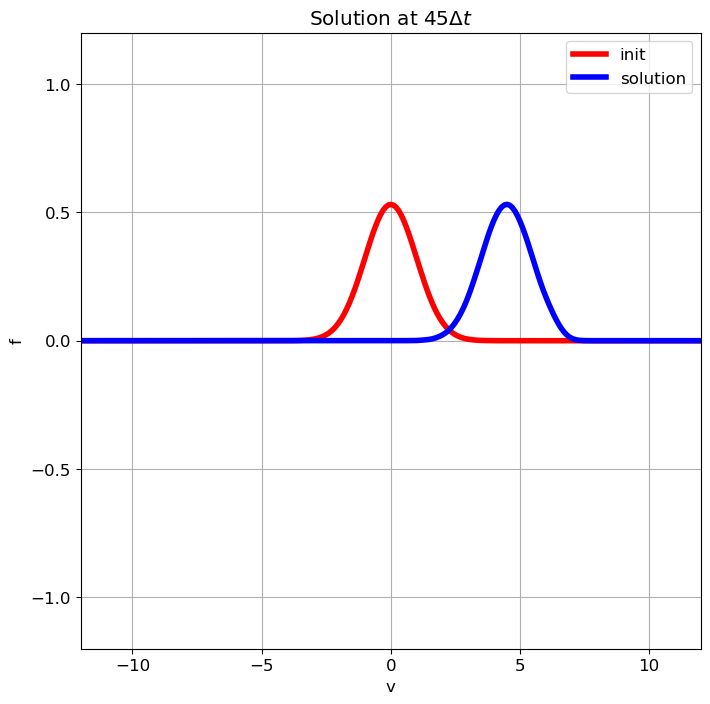}
&\includegraphics[scale = 0.275]{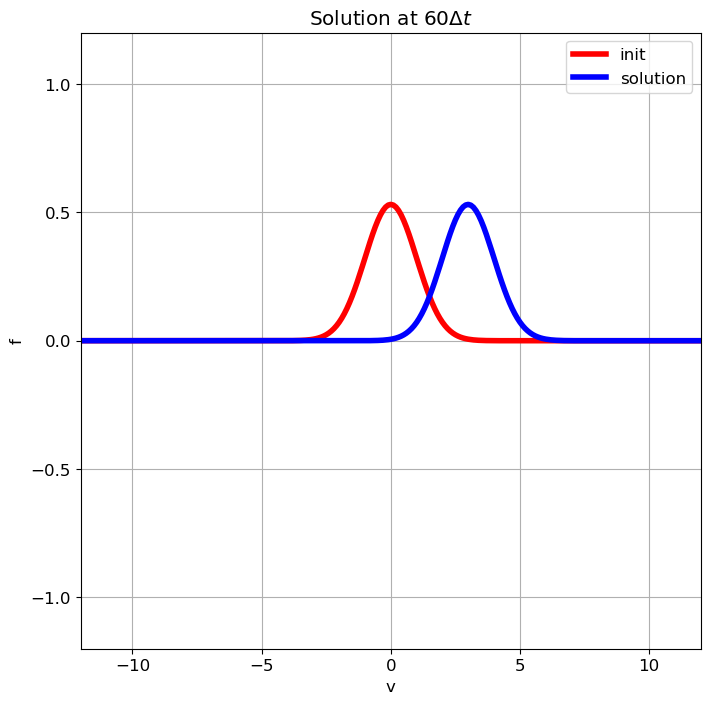} \\
 \includegraphics[scale = 0.275]{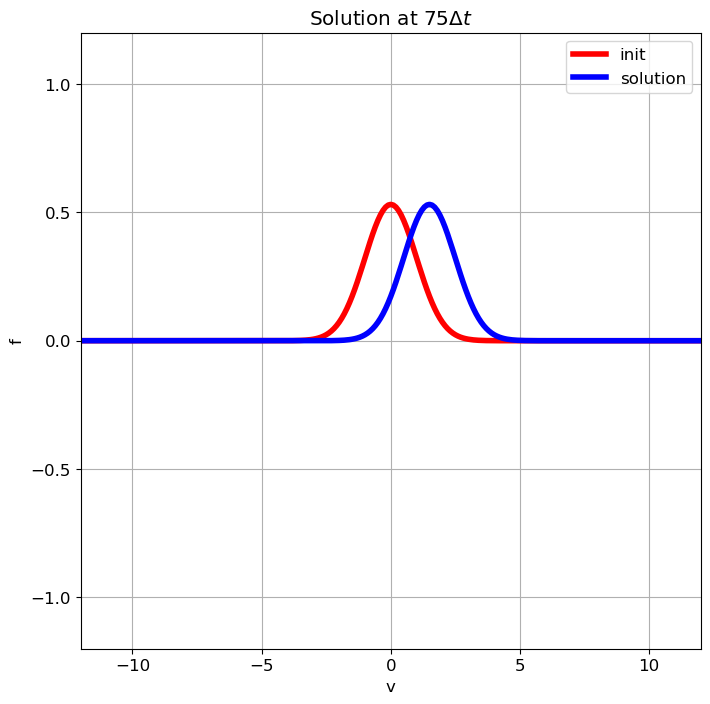}
&\includegraphics[scale = 0.275]{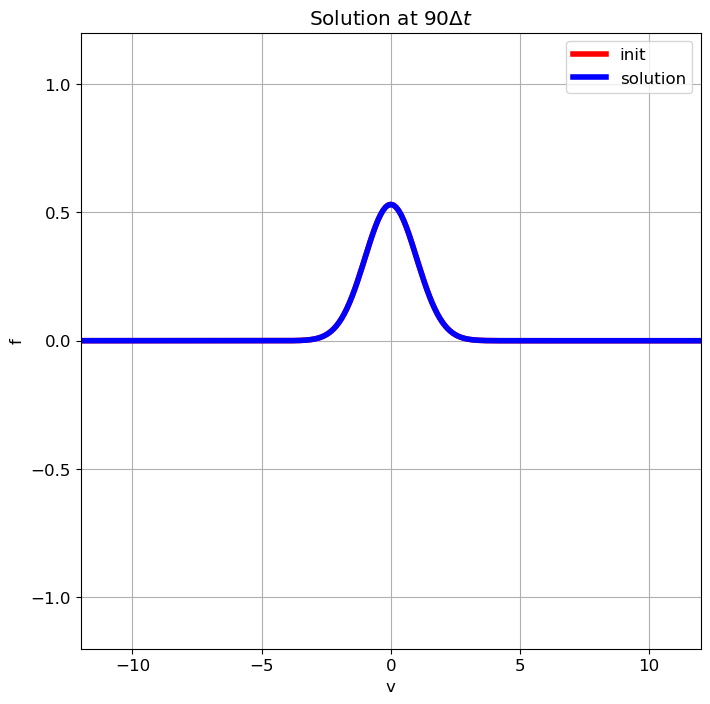}
 \end{tabular}
 \caption{Results of the advection test computed the scheme (\ref{eq:b8}) at time $t_0=0\Delta t$ to $t_5=90\Delta t$ ($N=64$ and $\Delta t=0.1$) with the second method (with $\varepsilon=10^{-10}$).}
 \label{fig:3}
\end{figure}
In Figure \ref{fig:3}  we plot the results where the matrix $\overline D^N$ is obtained with the method of Section \ref{sec:penal}.
The norm $  \left<U^N, A^N U^N \right> + \varepsilon \| U^N\|^2$ is rigorously constant one  time step after the other, as stated in Lemme \ref{lem:5.8}.
In both  cases, we observe excellent stability and accuracy.


It is of course possible to use the method with a constant electric field without reversing the sign. The results are  perfectly stable but the accuracy is more difficult to measure
because of the reasons explained in Remark \ref{rem:2.1}.


\subsection{Diocotron instability}

As outlined in Section \ref{sec:6}, 
our method can be employed for solving 
the Vlasov equation \eqref{eq:vlasov-poisson} coupled with the Poisson equation. 
The diocotron instability, observable in  magnetized low-density nonneutral plasmas with velocity shear, 
generates electron vortices akin to the Kelvin-Helmholtz fluidic shear instability. 
It arises when charge neutrality is disrupted, 
observed in scenarios like non-neutral electron beams and layers. 
The magnetic field's strength induces electron motion dominated by advection within the self-consistent $\mathbf{E} \times \mathbf{B}$ velocity field. 
The initial non-monotonic electron density profile creates an unstable $\mathbf{E} \times \mathbf{B}$ shear flow, 
resembling Kelvin-Helmholtz shear layer instability in fluid dynamics and the diocotron instability in beam and plasma physics. 
As this instability progresses nonlinearly, the initially axisymmetric electron density distribution distorts, 
resulting in discrete vortices and eventual breakup. 
This test case holds significance in both fundamental physics and practical applications, such as beam collimation.

The initial condition and the parameters are the same as those
in \cite{muralikrishnan},
with a uniform external magnetic field $\mathbf{B} = (0, 0, 5)$
applied along the $z$-axis within a domain of length $L = 22$.
Additionally, the external electric field is set to $0$ for this particular problem.
The initial condition is given by
$$
  f(t=0, r, v)
= \dfrac{C}{2\pi}
  \exp\left\{\frac{-|\mathbf{v}|^2}{2}\right\}
  \exp\left\{\frac{-(r-L/4)^2}{2(0.03L)^2}\right\}
,
$$
where $r = \sqrt{(x-L/2)^2 + (y-L/2)^2}$,
and the constant $C$ is selected to ensure that
the overall electron charge $Q_e$ equals $-400$.
The time integrator employs a time step of $\Delta t = 0.01$,
and the simulation is executed until the final time $T = 10.0$.

We consider three distinct cases. The first case is the original one without stabilization.
The second method is with the first method.
The third method is with the first method which implements the simplification explained in Remark \ref{rem:remi}, that is we nullify the first coefficients of the modified matrices.

We perform tests at $N=40$. For finite difference grid, we test a grid resolution of $64$.
We are interested in the way our new methods recover the quadratic stability. 
\begin{figure}[h!]
\centering
  \begin{tabular}{|ccc|}
    \hline
    & & \\
    \includegraphics[width=0.3\textwidth]{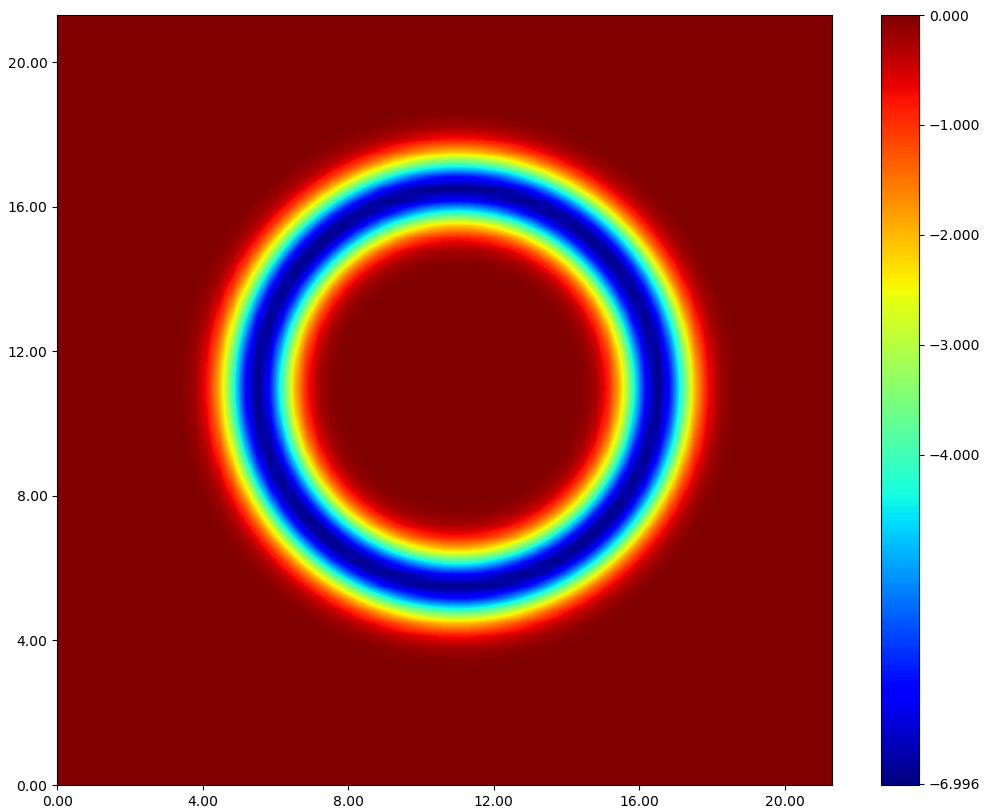}   &
    \includegraphics[width=0.3\textwidth]{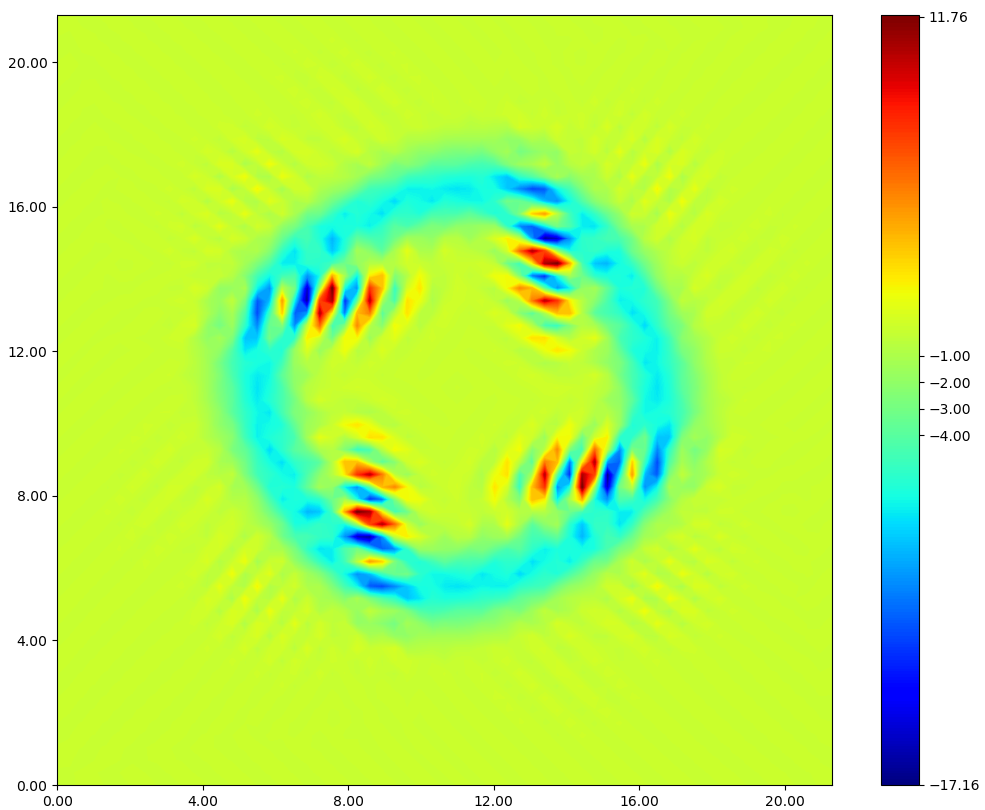} &
    \includegraphics[width=0.3\textwidth]{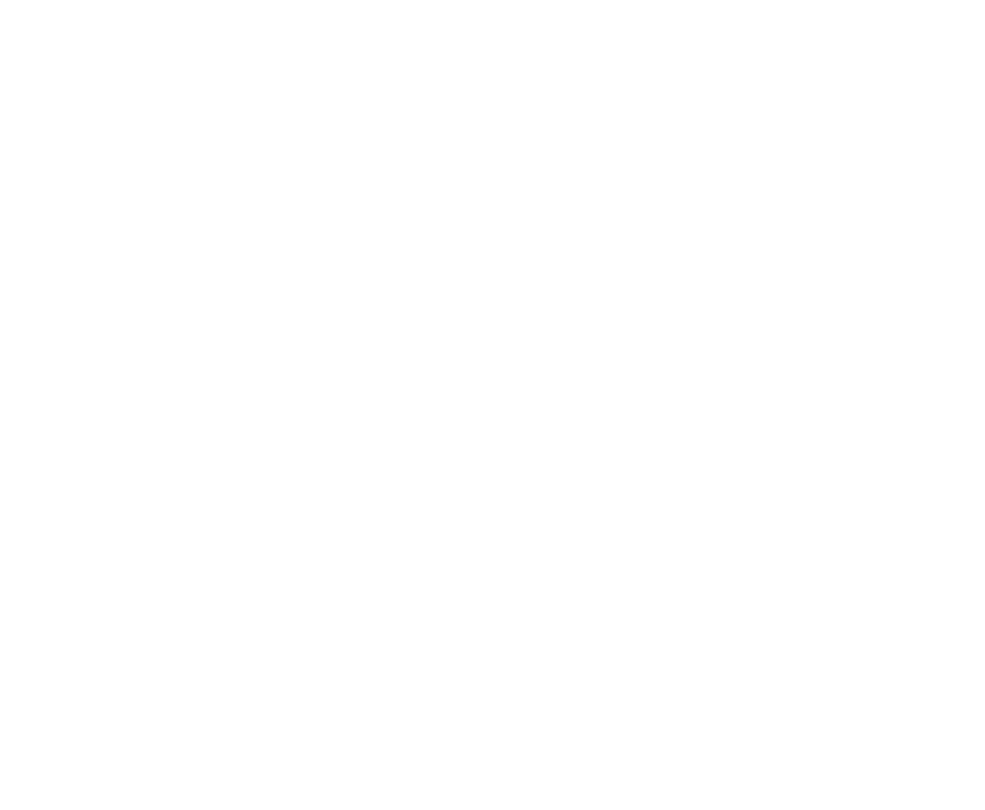}\\
    {$t=0$} & 
    {$t=3.8$} & 
    \\
    \hline
    & & \\
    \includegraphics[width=0.3\textwidth]{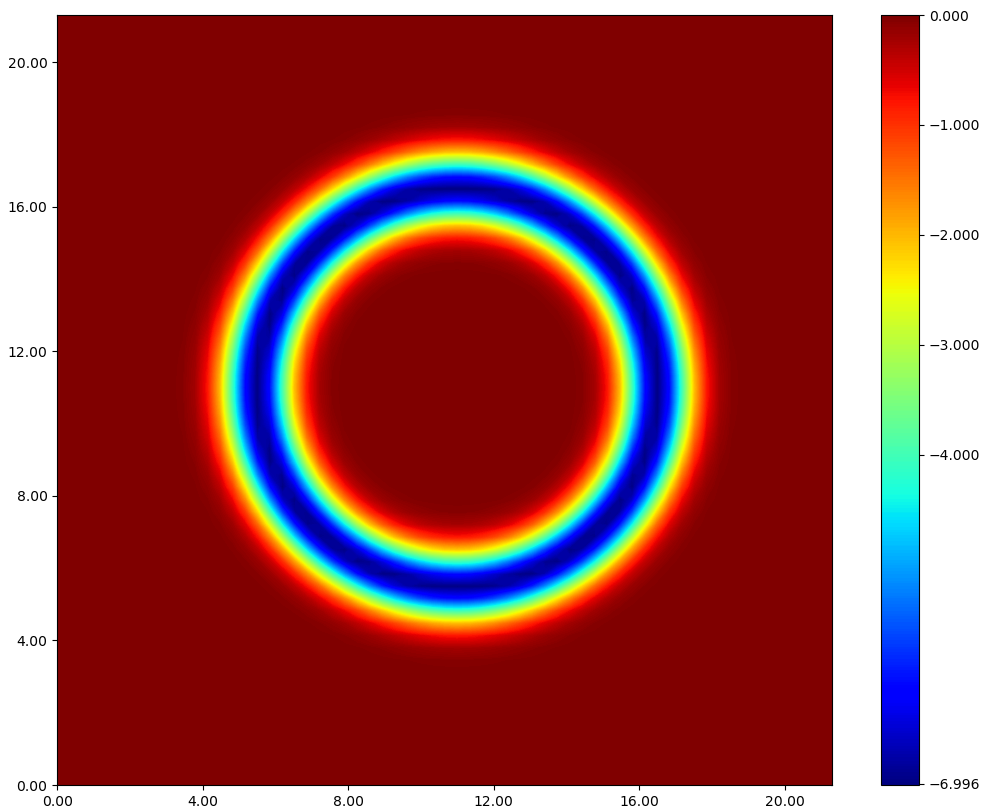}   &
    \includegraphics[width=0.3\textwidth]{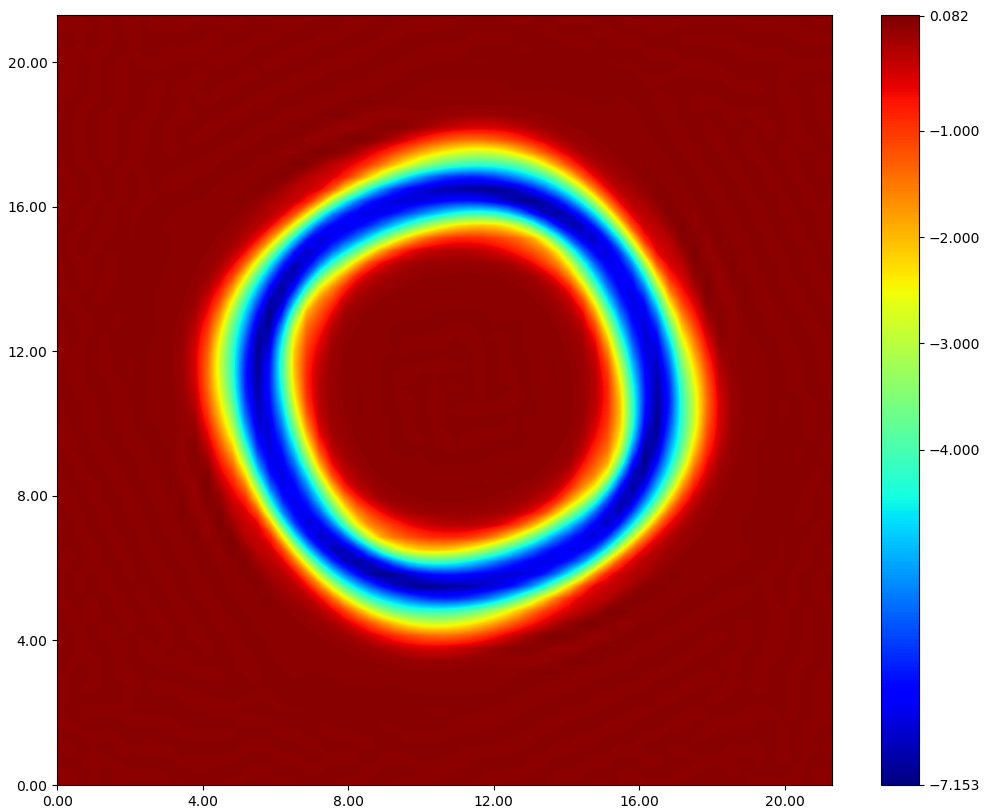} &
    \includegraphics[width=0.3\textwidth]{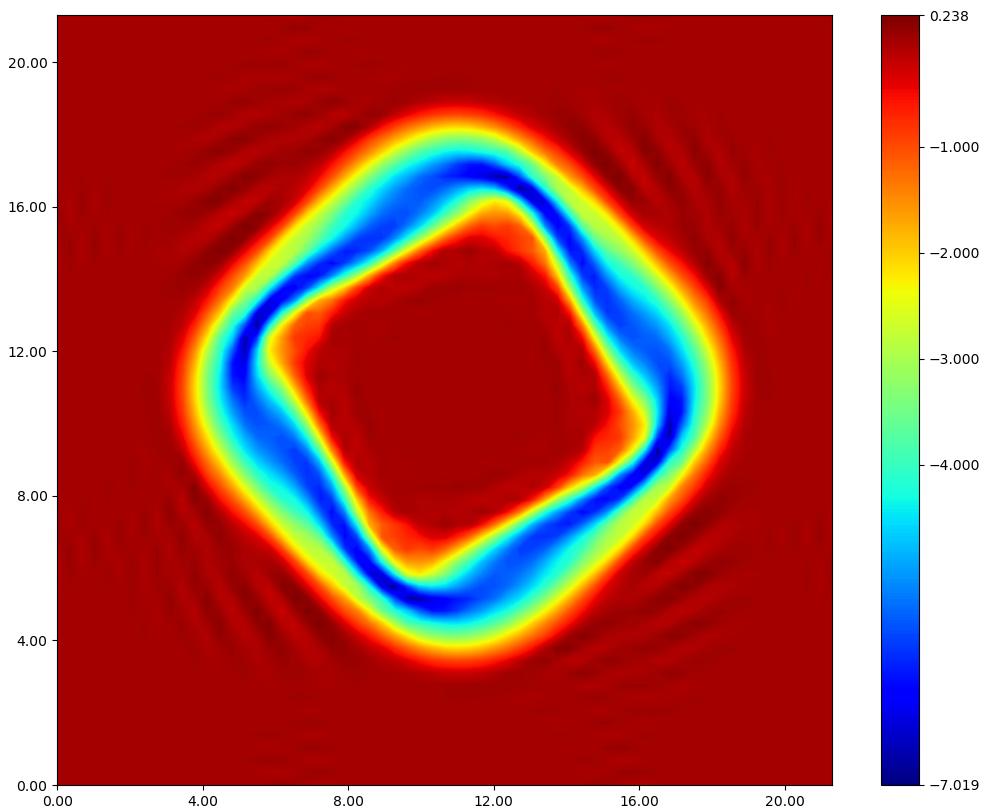}\\
    {$t=0$} & 
    {$t=3.8$} & 
    {$t=10$} \\
    \hline
    & & \\
    \includegraphics[width=0.3\textwidth]{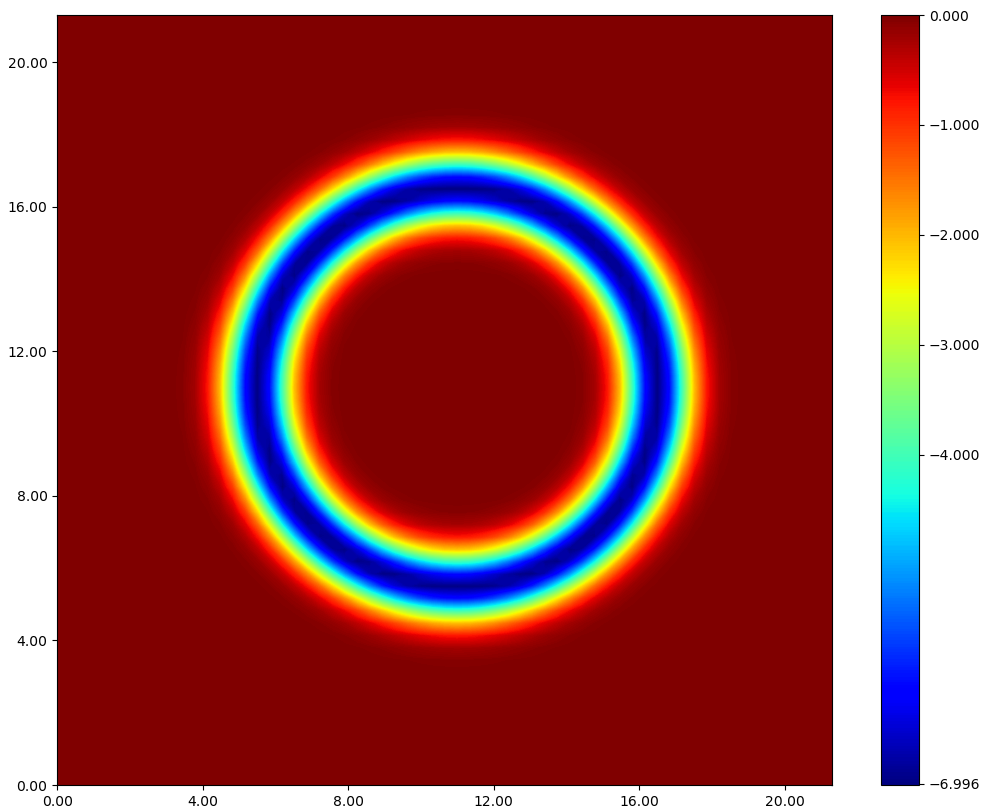}   &
    \includegraphics[width=0.3\textwidth]{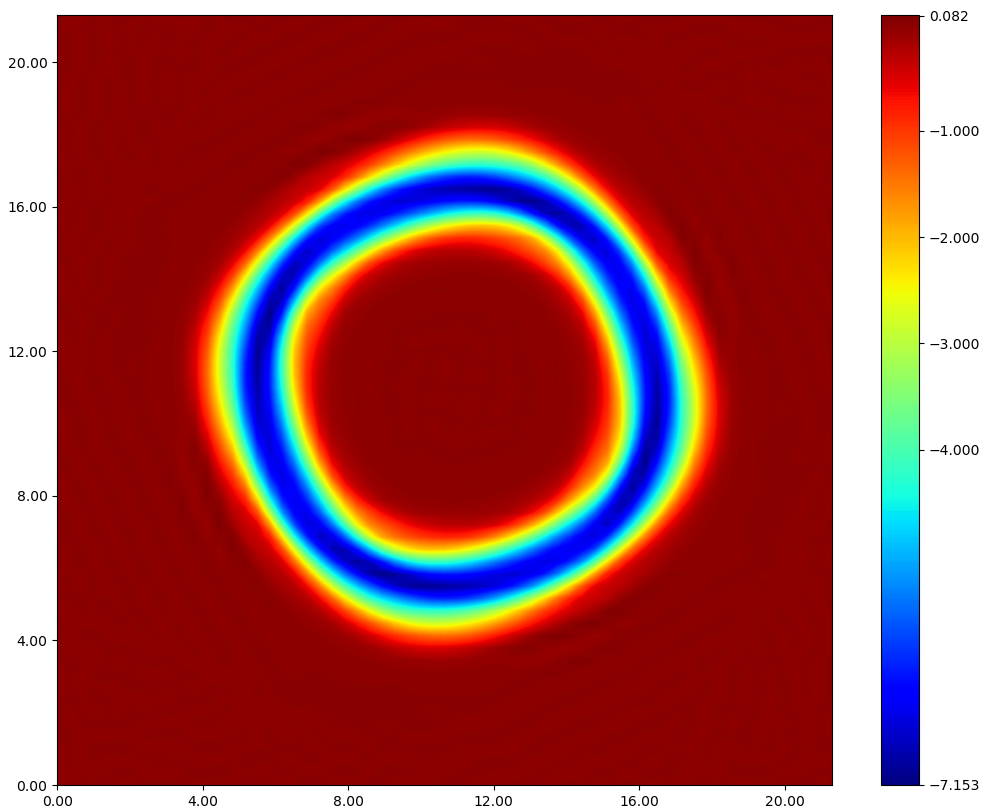} &
    \includegraphics[width=0.3\textwidth]{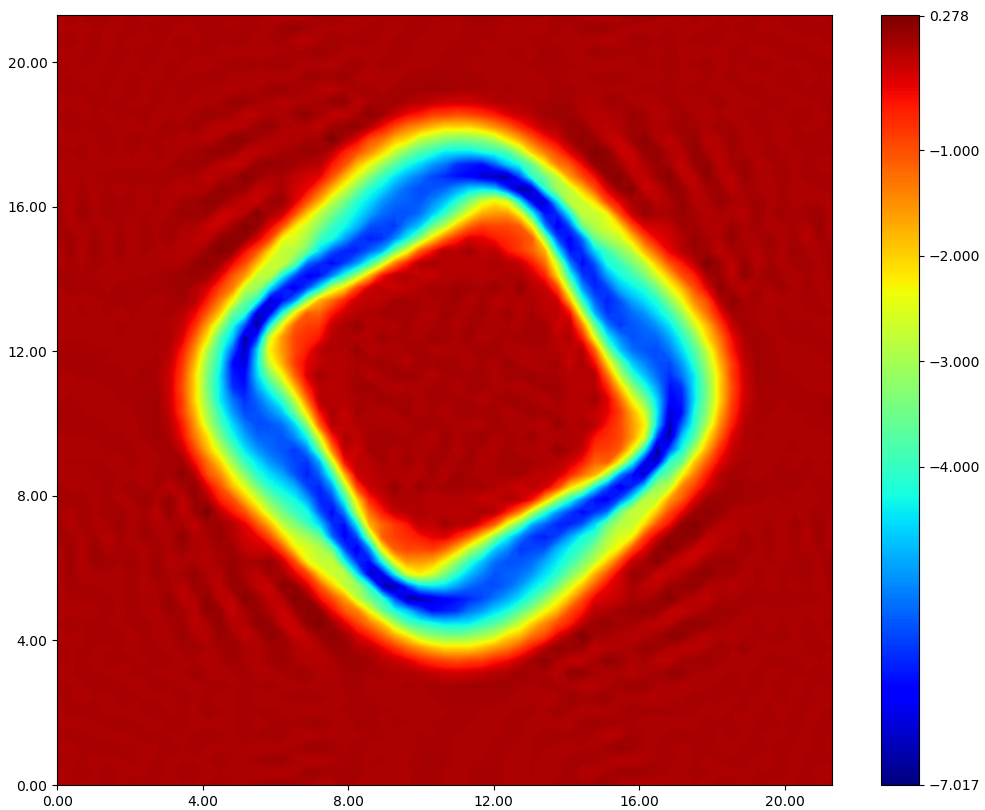}\\
    {$t=0$} & 
    {$t=3.8$} & 
    {$t=10$} \\
    \hline
  \end{tabular}
  \caption{2D diocotron instability: Evolution of electron charge density over time utilizing three distinct cases: the original method without stabilization (first row), the first method (second row), and the first method that implements the simplification explained in Remark \ref{rem:remi}, which is a numerically conservative method (third row). These visualisations are based on a $64^2$ mesh grid. Each figure's color bar displays the respective minimum and maximum values of the charge densities.}
\label{fig:diocotron}
\end{figure}

Fig. \ref{fig:diocotron} illustrates how the electron charge density changes over time,
for the unstabilized method (first row), the stabilized method (second row) and
the stabilized   method conservative in mass (third row).
From the first row one can observe the numerical instability ($t=3.80$),
and the blow-up.
From the second and third rows, one can observe that two methods are well stabilized  for all times.
The results in the third row closely align with the results presented in the second row, based on visual standards.
Conversely, none of the three methods enforces the positivity of the particle distribution function, resulting in minor oscillations around zero values.

\subsection{Two-stream instability}

We take the data of the two stream instability from \cite{flib:xong}.
The initial data is
$$
f_0(x,v)=\frac27 \left( 1+\cos kx + \alpha (\cos 2kx +\cos 3kx) /1.2 \right)(1+v^2) \frac1{\sqrt{2\pi }} e^{-v^2/2}
$$
with $\alpha=0.01$ and $k=0.5$. 
We perform tests at $N = 64$. For finite difference grid, we test a grid resolution of $32$.
And the time step is $\Delta t = 0.01$.
For this problem only two moments are non zero, which are 
\[
\left\{
\begin{aligned}
&u_0(x)= \frac{12}        7\left( 1+\cos kx + \alpha (\cos 2kx +\cos 3kx) /1.2 \right) \\ 
&u_2(x)= \frac{10\sqrt 2} 7\left( 1+\cos kx + \alpha (\cos 2kx +\cos 3kx) /1.2 \right), 
\end{aligned}
\right.
\]
and all other moments vanish.
\begin{figure}[h!]
\centering
  \begin{tabular}{|c|c|}
    \hline
    & \\
    \includegraphics[width=5.5cm]{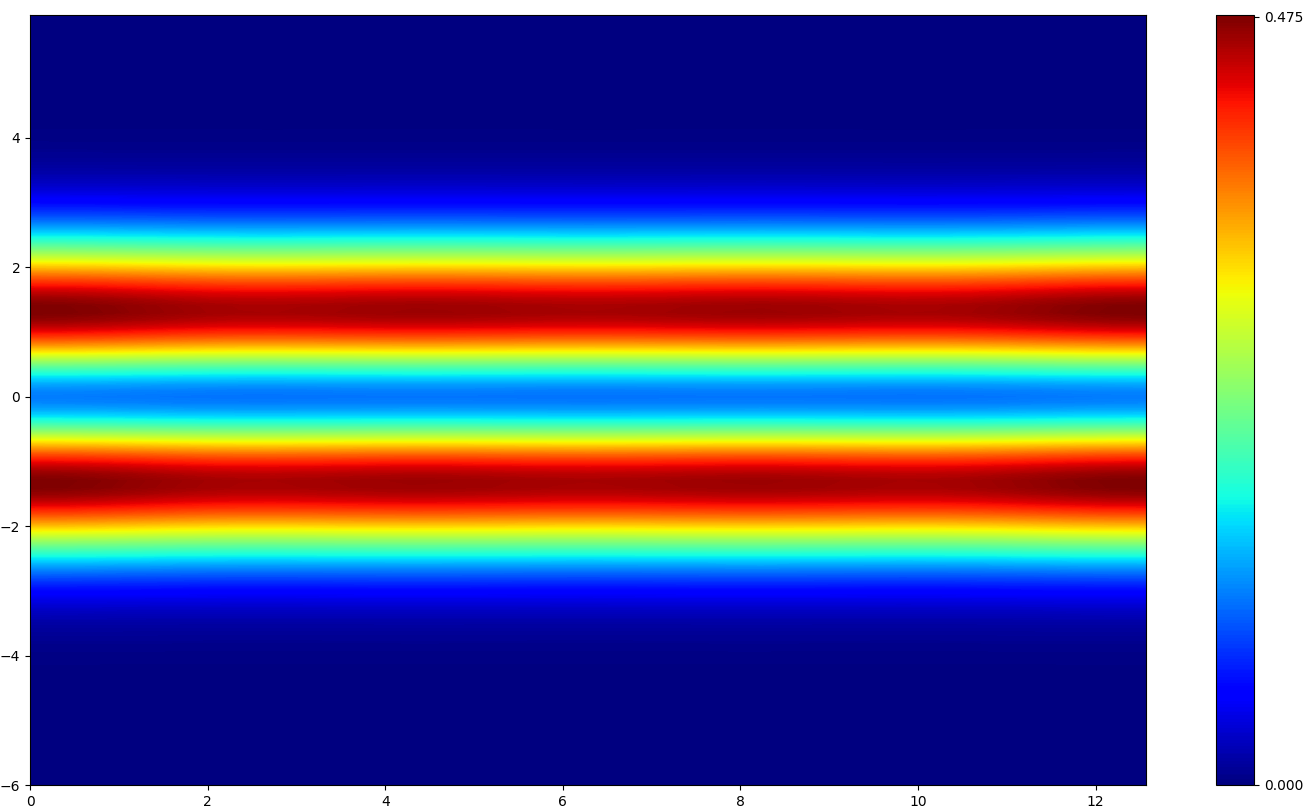} &
    \includegraphics[width=5.5cm]{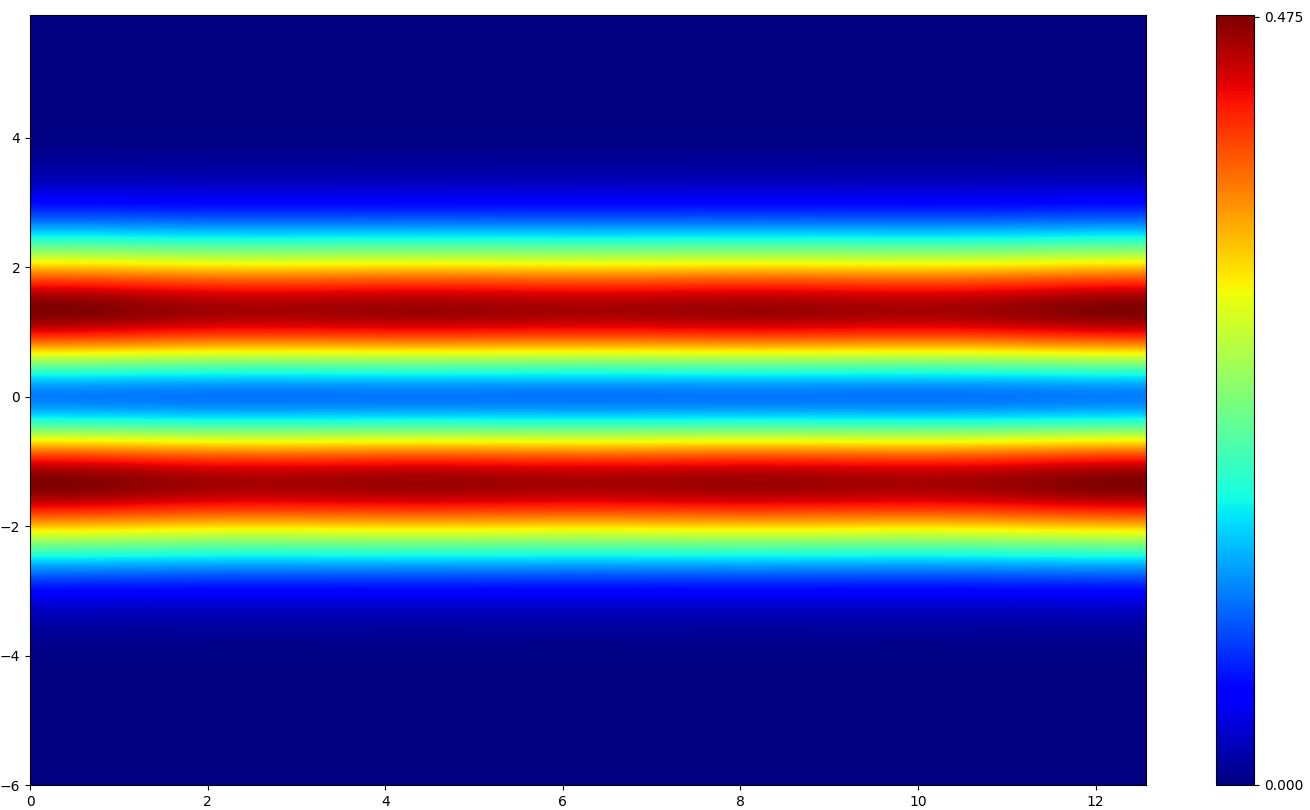} \\ 
    {$t=0$} & 
    {$t=0$} \\
    \includegraphics[width=5.5cm]{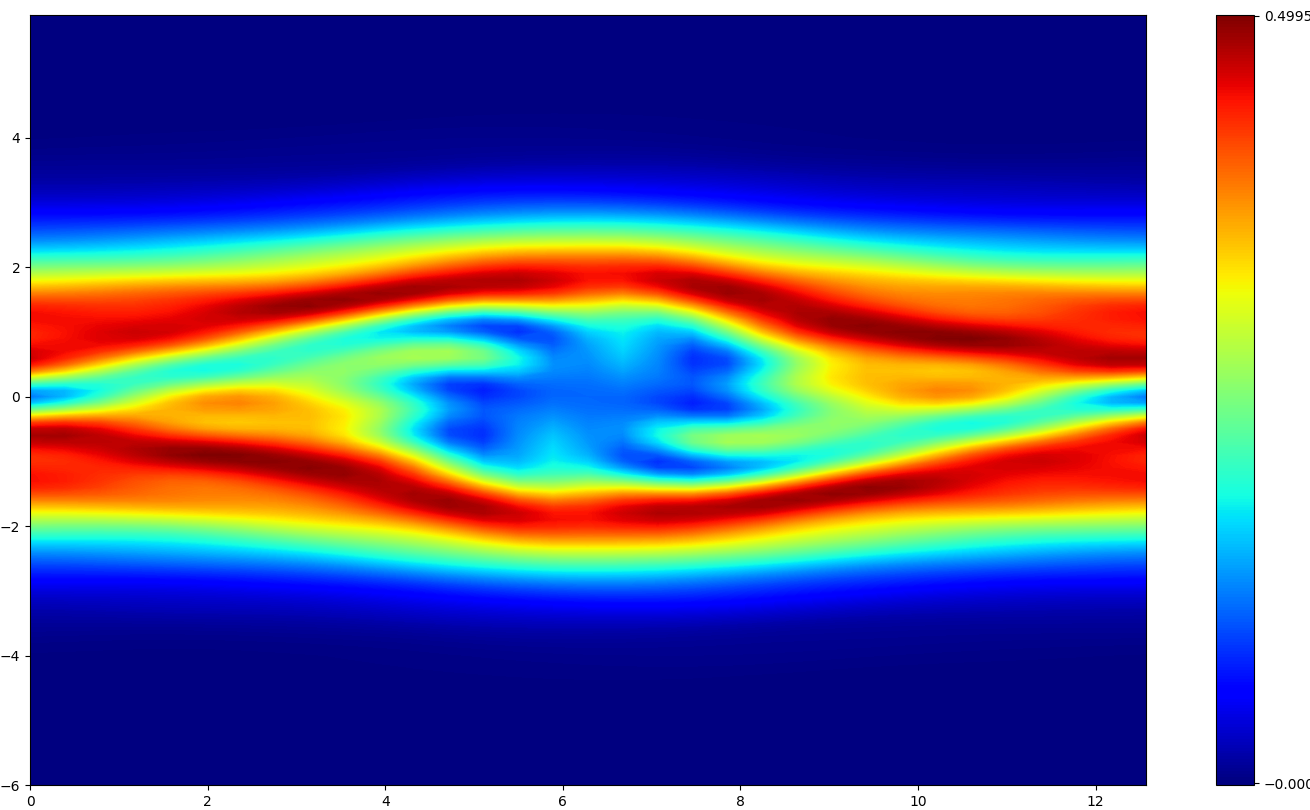} & 
    \includegraphics[width=5.5cm]{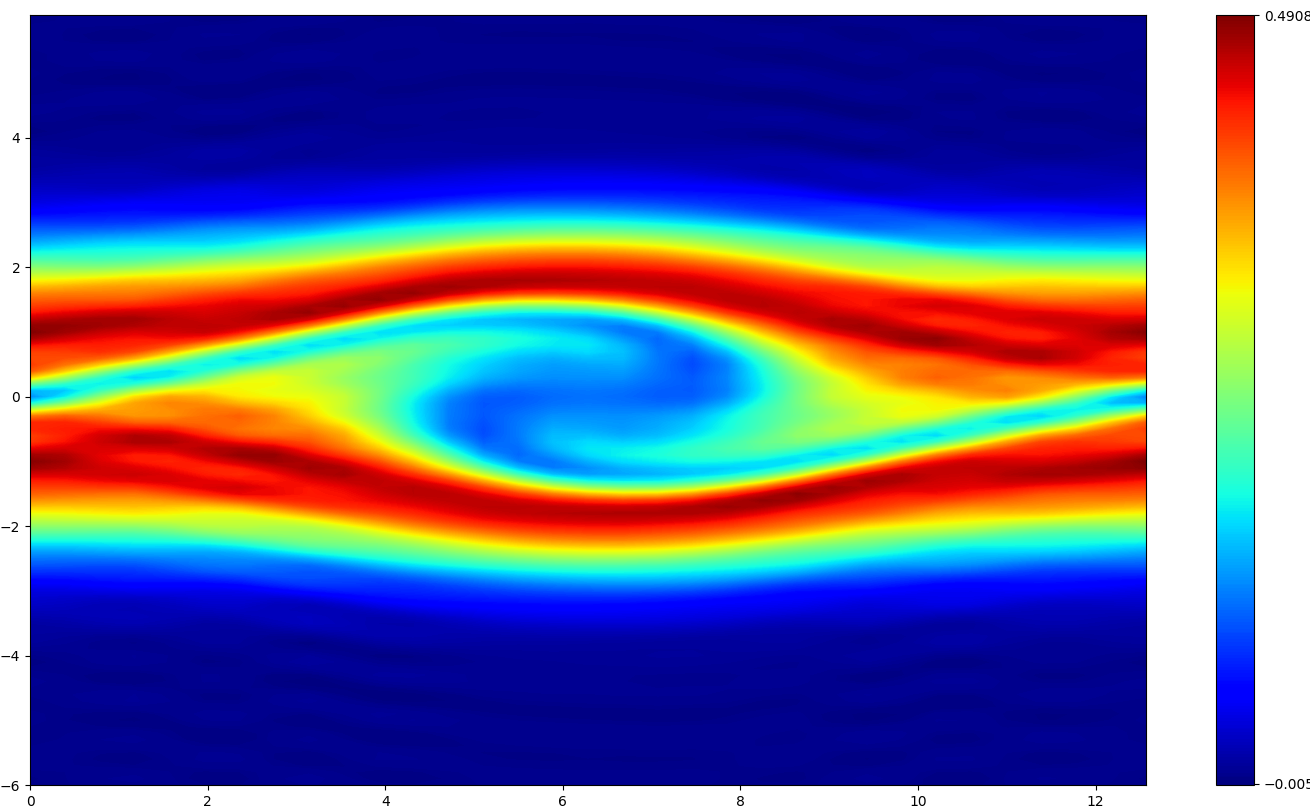} \\
    {$t=20$} & 
    {$t=20$} \\
    \hline
  \end{tabular}
  \caption{Density function at time $t=0$ and $t=20$. On the left the original method. On the right the stabilization with the first method.}
\label{fig:ts1}
\end{figure}

The results are shown in Figure \ref{fig:ts1} and \ref{fig:t2}.
The density function calculated at time $t=20$ is represented together with the history of the potential energy with respect to the time variable.
Up to a multiplicative constant, the potential energy is in accordance with the result from  \cite{flib:xong} in both cases.
In our opinion our numerical results  illustrate   that the stabilization of the asymmetric Hermite functions has a potential   for the  computation  of such non linear dynamics without any post-processing
or filtering of the numerical results. 
\begin{figure}[h!]
\centering
\includegraphics[width=\textwidth]{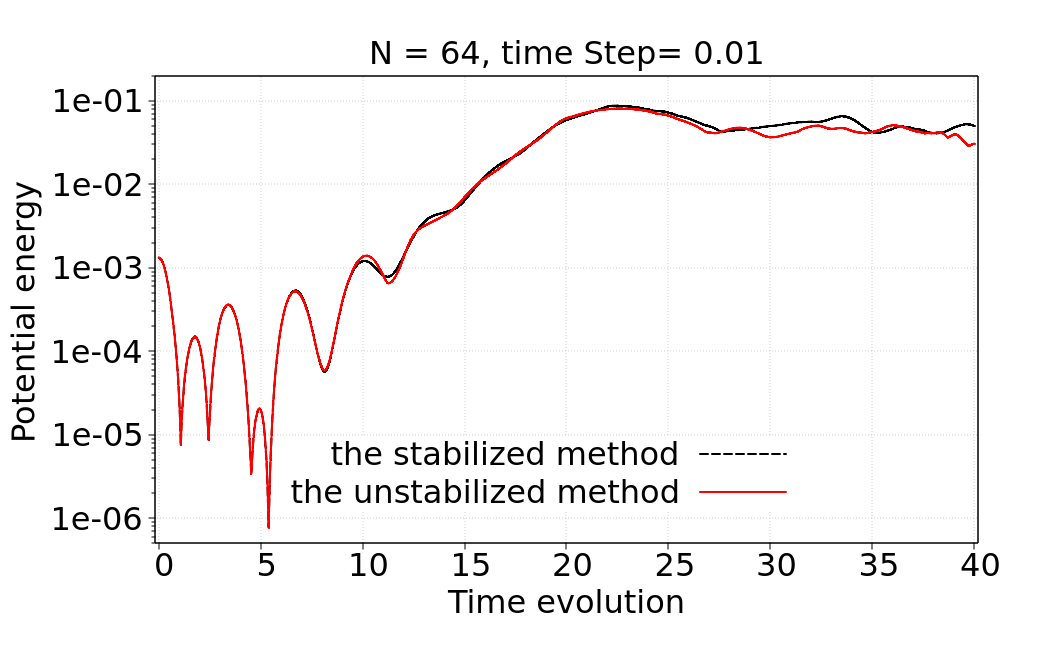} 
\caption{History of the potential energy with respect to the time variable.}
\label{fig:t2}
\end{figure}

{\bf Acknowledments.} The authors extend their heartfelt gratitude to Sever Hirstoaga and Fr\'ed\'erique Charles whose remarks were essential during the elaboration stage of this work.

{\bf Funding.} This study has been supported by ANR MUFFIN ANR-19-CE46-0004.

{\bf Data Availability.} The paper has no associated data

\section*{Declarations}

{\bf Conflict of interest.} The authors declare that they have no conflict of interest.

\appendix

\end{document}